\documentclass[10pt]{amsart}
\usepackage{amssymb, enumitem}
\usepackage[all]{xy}
\usepackage{hyperref, aliascnt}
\usepackage{mathtools}
\usepackage{bbm}
\def\today{\number\day\space\ifcase\month\or   January\or February\or
   March\or April\or May\or June\or   July\or August\or September\or
   October\or November\or December\fi\   \number\year}

\theoremstyle{definition}
\newtheorem{lma}{Lemma}[section]

\newaliascnt{thmCt}{lma}
\newtheorem{thm}[thmCt]{Theorem}
\aliascntresetthe{thmCt}

\newaliascnt{corCt}{lma}
\newtheorem{cor}[corCt]{Corollary}
\aliascntresetthe{corCt}

\newaliascnt{propCt}{lma}
\newtheorem{prop}[propCt]{Proposition}
\aliascntresetthe{propCt}

\newtheorem*{thm*}{Theorem}
\newtheorem*{cor*}{Corollary}
\newtheorem*{prop*}{Proposition}

\newaliascnt{pgrCt}{lma}

\aliascntresetthe{pgrCt}

\newaliascnt{dfCt}{lma}
\newtheorem{df}[dfCt]{Definition}
\aliascntresetthe{dfCt}

\newaliascnt{remCt}{lma}
\newtheorem{rem}[remCt]{Remark}
\aliascntresetthe{remCt}

\newaliascnt{remsCt}{lma}

\aliascntresetthe{remsCt}

\newaliascnt{egCt}{lma}

\aliascntresetthe{egCt}

\newaliascnt{egsCt}{lma}
\newtheorem{egs}[egsCt]{Examples}
\aliascntresetthe{egsCt}

\newaliascnt{qstCt}{lma}

\aliascntresetthe{qstCt}

\newaliascnt{pbmCt}{lma}

\aliascntresetthe{pbmCt}

\newaliascnt{notaCt}{lma}

\aliascntresetthe{notaCt}

\newcommand{\beq}{\begin{equation}}
\newcommand{\eeq}{\end{equation}}
\newcommand{\beqa}{\begin{eqnarray*}}
\newcommand{\eeqa}{\end{eqnarray*}}
\newcommand{\bal}{\begin{align*}}
\newcommand{\eal}{\end{align*}}
\newcommand{\bi}{\begin{itemize}}
\newcommand{\ei}{\end{itemize}}
\newcommand{\be}{\begin{enumerate}}
\newcommand{\ee}{\end{enumerate}}

\newcommand{\ep}{\varepsilon}

\newcommand{\C}{{\mathbb{C}}}
\newcommand{\N}{{\mathbb{N}}}

\newcommand{\Hi}{{\mathcal{H}}}
\newcommand{\K}{{\mathcal{K}}}

\newcommand{\U}{{\mathcal{U}}}

\newcommand{\T}{{\mathbb{T}}}
\newcommand{\D}{{\mathcal{D}}}

\newcommand{\Ot}{{\mathcal{O}_2}}
\newcommand{\OI}{{\mathcal{O}_{\I}}}

\newcommand{\dimRok}{{\mathrm{dim}_\mathrm{Rok}}}

\newcommand{\dimnuc}{{\mathrm{dim}_\mathrm{nuc}}}
\newcommand{\dr}{{\mathrm{dr}}}

\newcommand{\id}{{\mathrm{id}}}

\newcommand{\dist}{{\mathrm{dist}}}

\newcommand{\Aut}{{\mathrm{Aut}}}

\newcommand{\Ad}{{\mathrm{Ad}}}

\newcommand{\Ann}{{\mathrm{Ann}}}

\newcommand{\ca}{$C^*$-algebra}

\newcommand{\uca}{unital $C^*$-algebra}

\newcommand{\Rp}{Rokhlin property}

\newcommand{\I}{\infty}

\title[]{Crossed products by compact group actions with the Rokhlin property}

\author{Eusebio Gardella}

\date{\today}

\thanks{}

\address{Mathematisches Institut, Fachbereich Mathematik und Informatik der
Universit\"at M\"unster, Einsteinstrasse 62, 48149 M\"unster, Germany.}

\email[]{gardella@uni-muenster.de}


\begin{document}

\begin{abstract}
We present a systematic study of the structure of crossed products and fixed point
algebras by compact group
actions with the Rokhlin property on not necessarily unital $C^*$-algebras.
Our main technical result is the existence of an approximate
homomorphism from the algebra to its subalgebra of fixed points, which is a left inverse
for the canonical inclusion. Upon combining this with results regarding local approximations,
we show that a number of classes characterized by inductive limit decompositions with weakly semiprojective
building blocks, are closed under formation of crossed products by such actions.
Similarly, in the presence of the Rokhlin property, if the algebra has any of the following properties,
then so do the crossed product and the fixed point algebra: being a Kirchberg algebra, being simple and
having tracial rank zero or one, having real rank zero, having
stable rank one, absorbing a strongly self-absorbing $C^*$-algebra, satisfying the Universal Coefficient
Theorem (in the simple, nuclear case), and being weakly semiprojective.
The ideal structure of crossed products and fixed point algebras by Rokhlin actions is also studied.

The methods of this paper unify, under a single conceptual approach, the work of a number of
authors, who used rather different techniques. Our methods yield
new results even in the well-studied case of finite group actions with the Rokhlin property.
\end{abstract}

\maketitle
\tableofcontents

\section{Introduction}

The Rokhlin property first appeared in the late 1970's and early 1980's,
in work of Fack and Mar\'echal \cite{FacMar1}, Kishimoto \cite{Kis_AutUHF}, and Herman and Jones
\cite{HerJon_period} on cyclic group actions on UHF-algebras, and in the work of Herman and Ocneanu
\cite{HerOcn_StabInteg} on integer actions on UHF-algebras.

In \cite{Izu_RpI}, Izumi provided a formal definition of the Rokhlin property for an arbitrary finite group action
on a \uca.
His classification theorems for Rokhlin actions (\cite{Izu_RpI}, \cite{Izu_RpII}) are among the major
results in the study of finite group actions.

In a different direction, Izumi \cite{Izu_RpI}, Hirshberg and Winter
\cite{HirWin_Rp}, Phillips \cite{Phi_tracialFirst}, Osaka and Phillips \cite{OsaPhi_CPRP}, and Pasnicu and Phillips
\cite{PasPhi_PermProp}, explored the structure of crossed products by finite group actions with the Rokhlin property on
unital \ca s, while Santiago \cite{San_RpNonunital} addressed similar questions in the non-unital case.
The questions and problems addressed in each of these works are different, and consequently the approaches
used by the above mentioned authors are substantially distinct in some cases.

In \cite{HirWin_Rp}, Hirshberg and Winter also introduced the \Rp\ for a compact group action on a unital \ca,
and their definition coincides with Izumi's in the case of finite groups. They showed that approximate divisibility
and $\mathcal{D}$-stability, for a strongly self-absorbing $C^*$-algebra $\mathcal{D}$, are preserved under
formation of crossed product by compact group actions with the \Rp. Extending the results of \cite{Phi_tracialFirst},
\cite{OsaPhi_CPRP}, and \cite{PasPhi_PermProp} to the case of arbitrary compact groups requires new insights, since the
main technical tool in all of these works (Theorem~3.2 in \cite{OsaPhi_CPRP}) seems not to have a satisfactory
analog in the compact group case.

In this paper, we extend the definition of Hirshberg-Winter to actions of compact groups on $\sigma$-unital
\ca s, and generalize the results on finite group actions with the \Rp\ of the above mentioned papers
to the case of compact group actions. Our contribution is two-fold. First, most of the results we prove here
were known only in some special cases (mostly for finite or circle group actions; see
\cite{Gar_Kir1} and \cite{Gar_Kir2} for the
circle case), and some of them had not been noticed even in the context of finite groups. Additionally,
we do not require our \ca s to be unital, unlike in \cite{Izu_RpI}, \cite{HirWin_Rp}, \cite{OsaPhi_CPRP}, or
\cite{PasPhi_PermProp}. Finally,
our methods represent a uniform treatment of the study of crossed products by actions with
the Rokhlin property, where the attention is shifted from the crossed product itself, to the algebra of fixed
points.

\vspace{0.3cm}

Our results can be summarized as follows (the list is not exhaustive). We point out that
(14) below was first obtained, with different techniques and for unital $C^*$-algebras,
by Hirshberg and Winter as part (1) of
Corollary~3.4 in \cite{HirWin_Rp}. Also, (10) and (14) were obtained in \cite{Gar_CptRok}.

\begin{thm*}
The following classes of $\sigma$-unital \ca s are closed under formation of crossed products and passage
to fixed point algebras by actions of second-countable compact groups with the \Rp:
\be\item Simple $C^*$-algebras (\autoref{cor:Simple}). More generally, the ideal structure can be completely determined (\autoref{thm:IdealStruct});
\item $C^*$-algebras that are direct limits of certain weakly semiprojective \ca s
(\autoref{thm:RpPreservesWSPAQ}). This includes
UHF-algebras (or matroid algebras), AF-algebras, AI-algebras,
A$\T$-algebras, countable inductive limits of one-dimensional NCCW-complexes, and several other classes
(\autoref{cor:ParticularClasses});
\item Kirchberg algebras (\autoref{cor:Kir});
\item Simple $C^*$-algebras with tracial rank at most one (\autoref{thm:TAF});
\item Simple, separable, nuclear \ca s satisfying the Universal Coefficient Theorem (\autoref{thm:UCT});
\item $C^*$-algebras with nuclear dimension at most $n$, for $n\in\N$ (\autoref{thm:DimNucDR});
\item $C^*$-algebras with decomposition rank at most $n$, for $n\in\N$ (\autoref{thm:DimNucDR});
\item $C^*$-algebras with real rank zero (\autoref{prop:rrz_sro});
\item $C^*$-algebras with stable rank one (\autoref{prop:rrz_sro});
\item $C^*$-algebras with strict comparison of positive elements (Corollary~3.19 in~\cite{Gar_CptRok});
\item $C^*$-algebras whose order on projections is determined by traces (\autoref{prop:OrdPjnTraces});
\item (Not necessarily simple) purely infinite $C^*$-algebras (\autoref{prop:PI});
\item Separable $\mathcal{D}$-absorbing \ca s, for a strongly self-absorbing \ca\ $\mathcal{D}$
(\autoref{thm:SSApreserved});
\item $C^*$-algebras whose $K$-groups are either: trivial, free, torsion-free, torsion, or finitely generated
(Corollary~3.4 in~\cite{Gar_CptRok});
\item Weakly semiprojective \ca s (\autoref{prop:WksjPreserved}).
\ee\end{thm*}

Our work yields new results even in the case of finite groups. For example,
in (14) above, we do not require the algebra $A$ to be simple, unlike in
Theorem~3.13 of \cite{Izu_RpI}. In addition, the classes of \ca s considered in
\autoref{thm:RpPreservesWSPAQ} may consist of simple \ca s, unlike in Theorem~3.5 in
\cite{OsaPhi_CPRP} (we also do not impose any conditions regarding corners of our algebras).
Additionally, in
\autoref{prop:WksjPreserved}, we show that the inclusion $A^\alpha\to A$ is sequence algebra extendible
(\autoref{df:SeqAlgExt}) whenever $\alpha$ has the \Rp, and hence weak semiprojectivity passes from $A$ to $A^\alpha$. Our conclusion
seems not to be obtainable with the methods developed in \cite{OsaPhi_CPRP} and related works,
since it is not in general true that a corner of a weakly semiprojective \ca\ is weakly semiprojective.

Given that our results all follow as easy consequences of our main technical observation,
\autoref{thm:ApproxHom}, which allows us to deal with the non-unital case as well,
we believe that this paper unifies the work of a number of authors, who used
rather different methods, under a single systematic and conceptual approach.

In this paper, we take $\N=\{1,2,\ldots\}$.

\vspace{0.3cm}

\textbf{Acknowledgements.} The author is grateful to Chris Phillips for a number of helpful
conversations regarding averaging processes. He also thanks Hannes Thiel for conversations
on the Cuntz semigroup and local approximations, and Juan Pablo Lago for
his support and encouragement. Finally, he thanks the referee for a number of comments and
suggestions that improved the quality of this work, and in particular for suggesting a simpler
proof of \autoref{thm:ApproxHom}.

\section{An averaging process}

We begin introducing some useful notation and terminology.

\subsection{Central sequence algebras and Rokhlin property}

Given a \ca\ $A$, let $\ell^\I(\N,A)$ denote the set of all bounded sequences in $A$ with the supremum norm
and pointwise operations. Then $\ell^\I(\N,A)$ is a \ca, and it is unital if $A$ is $\sigma$-unital, since
any countable approximate unit for $A$ determines a unit for $\ell^\I(\N,A)$.
Set
$$c_0(\N,A)=\{(a_n)_{n\in\N}\in\ell^\I(\N,A)\colon \lim_{n\to\I}\|a_n\|=0\}.$$
Then $c_0(\N,A)$ is an ideal in $\ell^\I(\N,A)$, and we denote the quotient
$\ell^\I(\N,A)/c_0(\N,A)$ by $A_\I$.
We write $\eta_A\colon \ell^\I(\N,A)\to A_\I$ for the quotient map. We identify $A$ with the subalgebra of
$\ell^\I(\N,A)$ consisting of the constant sequences, and with a subalgebra of $A_\I$ by
taking its image under $\eta_A$. If $D$ is any subalgebra of $A$, then $A_\I\cap D'$ denotes the relative
commutant of $D$ inside of $A_\I$.

\begin{df}\label{df:CtralSeq} For a subalgebra $D\subseteq A$, write $\Ann(D,A_\I)$ for the annihilator of $D$ in $A_\I$, which is an ideal in $A_\I\cap D'$.
Following Kirchberg (\cite{Kir_CentralSeq}), we set
\[F(D,A)=A_\I\cap D'/\Ann(D,A_\I),\]
and write $\kappa_{D,A}\colon A_\I\cap D'\to F(D,A)$ for the quotient map. When $D=A$, we abbreviate $F(A,A)$ and $\kappa_{D,A}$
to $F(A)$ and $\kappa_A$.

If $\alpha\colon G\to\Aut(A)$ is an action of $G$ on $A$, and $D$ is an $\alpha$-invariant subalgebra of $A$, then there
are (not necessarily continuous) actions
of $G$ on $\ell^\I(\N,A)$, on $A_\I$, on $A_\I\cap D'$ and on $F(D,A)$, respectively denoted, with a slight
abuse of notation, by $\alpha^\I$, $\alpha_\I$, $\alpha_\I$ and $F(\alpha)$. Following Kishimoto (\cite{Kis_flows}),
we set
$$\ell^\I_\alpha(\N,A)=\{a\in \ell^\I(\N,A)\colon g\mapsto \alpha^\I_g(a) \mbox{ is continuous}\}.$$
We also set $A_{\I,\alpha}=\eta_A(\ell^\I_\alpha(\N,A))$ and $F_{\alpha}(A)=\kappa_{D,A}(A_{\I,\alpha}\cap D')$.\end{df}

By construction, $A_{\I,\alpha}$ and $F_{\alpha}(D,A)$ are invariant under $\alpha_\I$ and $F(\alpha)$, so the restrictions
of $\alpha_\I$ and $F(\alpha)$ to $A_{\I,\alpha}$ and $F_{\alpha}(D,A)$, which we also denote by $\alpha_\I$ and $F(\alpha)$, are continuous.

If $G$ is a locally compact group, we denote by
$\verb'Lt'\colon G\to\Aut(C_0(G))$ the action induced by left translation of $G$ on itself.

The following generalizes Definition~3.2 of \cite{HirWin_Rp} to the $\sigma$-unital setting.
(See Definition~3.2 in~\cite{Naw_RpNonunital} for the case of finite groups.) It should
also be compared with Definition~1.6 in~\cite{Sza_cRp}.

\begin{df}\label{df:Rp}
Let $A$ be a $\sigma$-unital \ca, let $G$ be a second-countable compact group,
and let $\alpha \colon G \to \Aut(A)$ be a continuous action. We say that $\alpha$
has the \emph{Rokhlin property} if for every separable $\alpha$-invariant subalgebra
$D\subseteq A$, there is an equivariant unital homomorphism
\[\varphi\colon (C(G),\texttt{Lt})\to (F_{\alpha}(D,A),F(\alpha)).\]\end{df}

A number of features of the Rokhlin property are studied in \cite{Gar_CptRok}. Here, we
shall focus on the crossed products and fixed point algebras, with emphasis on their
structure and classifiability.

We will repeatedly use the following fact, which is probably standard. Its proof can be found,
for example, in \cite{GHS_preparation}. For compact $G$, we identify $C(G,A)$ and
$C(G)\otimes A$ in the usual way.

\begin{prop}\label{prop:CGAequivisom}
Let $G$ be a compact group, let $A$ be a \ca, and let $\alpha\colon G\to\Aut(A)$ be an action.
Define a homomorphism $\sigma\colon C(G,A)\to C(G,A)$ by
$\sigma(a)(g)=\alpha_g(a(g))$ for $a\in C(G,A)$ and $g\in G$. Then
\[\sigma\colon (C(G,A),\texttt{Lt}\otimes\id_A)\to (C(G,A),\texttt{Lt}\otimes\alpha)\]
is an equivariant isomorphism.
\end{prop}

We need an easy lifting result. We thank Luis Santiago for pointing it out to us.

\begin{lma}\label{lma:LiftFromF}
Let $G$ be a locally compact group, let $C$ and $A$ be \ca s, and let $\gamma\colon G\to\Aut(C)$
and $\alpha\colon G\to \Aut(A)$ be actions, and let $D\subseteq A$ be an invariant subalgebra.
Give $C\otimes_{\mathrm{max}}A$ the diagonal $G$-action.
Suppose that there exists a unital equivariant homomorphism $\varphi\colon C\to F_\alpha(D,A)$, and
choose any function $\theta\colon C\to A_{\I,\alpha}\cap D'$ satisfying $\kappa_A\circ\theta=\varphi$.
Then there exists an equivariant homomorphism
\[\psi\colon C\otimes_{\mathrm{max}} D\to A_{\I,\alpha}\]
determined by $\psi(c\otimes d)=\theta(c)d$ for all $c\in C$ and all $a\in A$. Moreover, $\psi$ does
not depend on $\theta$.
\end{lma}
\begin{proof}
We check that $\psi$ is indeed a homomorphism.
Let $c_1,c_2\in C$ and $d_1,d_2\in A$ be given. Using that $\theta(c_1c_1)x=\theta(c_1)\theta(c_2)x$ for any $x\in D$ at
the second step, and that $\theta(C)$ commutes with $D$ at the third step, we get
\begin{align*} \psi(c_1c_2\otimes d_1d_1)&=\theta(c_1c_2)d_1d_1=\theta(c_1)\theta(c_2)d_1d_1=\theta(c_1)d_1\theta(c_2)d_2\\
 &=\psi(c_1\otimes d_1)\psi(c_2\otimes d_2),
\end{align*}
as desired. Finally, if $\widetilde{\theta}$ is another lift of $\varphi$, then clearly
$\widetilde{\theta}(c)d=\theta(c)d$ for
all $c\in C$ and all $a\in D$, which shows that $\psi$ does not depend on the lift of $\varphi$. \end{proof}

\subsection{First results on crossed product and the averaging process}
If a compact group $G$ acts on
a \ca\ $A$, then $A^G$ is naturally a corner in $A\rtimes G$ (see the Theorem in \cite{Ros_corner}),
even though $A$ is itself not in
general a subalgebra of $A\rtimes G$. (When $G$ is discrete, there
is a different way of regarding $A^G$ as a subalgebra of the crossed product, since $A$ always
sits inside $A\rtimes G$. When $G$ is finite, these two inclusions never agree when $G$ is not
trivial, and we will exclusively deal with the corner inclusion considered by Rosenberg.)

Using this corner inclusion, one can many times obtain
information about the fixed point algebra through the crossed product. However, since this
corner is not in general full, Rosenberg's theorem is not always useful if one is interested
in transferring structure from $A^G$ to $A\rtimes G$.
Saturation for compact group actions is the basic notion that allows one to do this, up to
Morita equivalence. The definition, which is essentially due to Rieffel, is as in Definition~7.1.4
in \cite{Phi_Book}. What we reproduce below is the equivalent formulation given in
Lemma~7.1.9 in \cite{Phi_Book}.
We point out that saturation is equivalent to the corner $A^G\subseteq A\rtimes G$ being full.

\begin{df}\label{df:saturation} (Definition~7.1.4 in \cite{Phi_Book}.)
Let $G$ be a compact group, let $A$ be a \ca, and let $\alpha\colon G\to\Aut(A)$ be an action.
We say that $\alpha$ is \emph{saturated}, if the set
\[\left\{f_{a,b}\colon G\to A; f_{a,b}(g)=a\alpha_g(b) \mbox{ for all } g\in G, \mbox{ with } a,b\in A\right\}\subseteq L^1(G,A,\alpha)\]
spans a dense subspace of $A\rtimes_\alpha G$.\end{df}

It is an easy exercise to check that if a compact group $G$ acts freely on a compact Hausdorff space $X$,
then the induced action on $C(X)$ is saturated. For this, it suffices to prove that the set
\[\left\{f_{a,b}\in C(G\times X)\colon
\begin{aligned}
& \ \ f_{a,b}(g,x)=a(x)b(g\cdot x) \mbox{ for all }\\
&  (g,x)\in G\times X, \mbox{ with } a,b\in C(X)
\end{aligned}
\right\}\]
spans a dense subset of $C(G\times X)$. This linear span is closed under multiplication and contains the constant
functions regardless of whether the action of $G$ is free or not, and it is easy to see that it separates the points
of $G\times X$ if and only if it is free. The claim then follows from the Stone-Weierstrass theorem. See Theorem~7.2.6
in~\cite{Phi_Book} for a more general result involving $C(X)$-algebras.

\begin{lma}\label{lma:C(G,D)Sat}
Let $\beta\colon G\to\Aut(C)$ be a saturated action of a compact group $G$ on
a nuclear \ca\ $C$, and let
$\id_D\colon G\to \Aut(D)$ denote the trivial action. Then the diagonal action
\[\gamma=\beta\otimes\id_D\colon G\to\Aut(C\otimes D)\]
is also saturated.
\end{lma}
\begin{proof}
Then there are canonical identifications
\[(C\otimes D)^\gamma\cong C^\beta\otimes D \ \mbox{ and }
\ (C\otimes D)\rtimes_\gamma G\cong (C\rtimes_\beta G)\otimes D.\]
Denote by $\iota_C\colon C^\beta\to C\rtimes_\beta G$ the canonical inclusion (see comments above \autoref{df:saturation}).
Observe that the saturation of $\beta$ is equivalent to the hereditary subalgebra generated by
$\iota_C(C^\beta)$ being all of $C\rtimes_\beta G$ (see Lemma~7.1.9 in \cite{Phi_Book}).
Under the above identifications, the inclusion
\[(C\otimes D)^\gamma\hookrightarrow (C\otimes D)\rtimes_\gamma G\]
corresponds to the map
\[\iota_C\otimes \id_D\colon C^\beta\otimes D\to (C\rtimes_\beta G)\otimes D.\]
Hence the image of $(C\otimes D)^\gamma$ generates all of $(C\otimes D)\rtimes_\gamma G$
as a hereditary subalgebra. We conclude that $\gamma$ is saturated.
\end{proof}

The following result will be used repeatedly throughout this paper.

\begin{prop}\label{prop:saturated}
Let $G$ be a second-countable compact group, let $A$ be a $\sigma$-unital \ca,
and let $\alpha\colon G\to\Aut(A)$ be an action with the \Rp. Then $\alpha$ is saturated.

In particular, the fixed point algebra and the crossed product by a compact group action with the
Rokhlin property are Morita equivalent, and thus stably isomorphic whenever the original algebra
is separable. \end{prop}
\begin{proof}
We begin by proving the statement for finite $G$ and unital, separable $A$, because we believe the reader will gain better
intuition from this particular case. Indeed, finiteness of $G$ allows one to construct the approximations
explicitly.

Suppose that $G$ is finite and $A$ is unital and separable. Fix $g\in G$, and denote by $u_g$ the canonical unitary
in the crossed product $A\rtimes_\alpha G$ implementing $\alpha_g$.
We claim that it is enough to show that $u_g$
is in the closed linear span of the functions $f_{a,b}$ from \autoref{df:saturation}. Indeed, if this
is the case, and if $x\in A$, then $xu_g$ also belongs to the closed linear span, and
elements of this form span $A\rtimes_\alpha G$.

For $n\in\N$, find projections $e_g^{(n)}\in A$, for $g\in G$, such that
\be
\item $\left\|\alpha_g(e^{(n)}_h)-e_{gh}^{(n)}\right\|<\frac{1}{n}$ for all $g,h\in G$; and
\item $\sum\limits_{g\in G}e^{(n)}_g=1$.
\ee

For $a,b\in A$, the function $f_{a,b}$ corresponds to the product $a\left(\sum\limits_{h\in G}\alpha_h(b)u_h\right)$.
Thus, for $n\in\N$ and $k\in G$, we have
\[f_{e^{(n)}_{gk},e^{(n)}_k}= e^{(n)}_{gk} \left(\sum_{h\in G}\alpha_h(e^{(n)}_k)u_h\right).\]
We use pairwise orthogonality of the projections $e_g^{(n)}$, for $g\in G$, at the third step, to get
\begin{align*}
\left\|f_{e^{(n)}_{gk},e^{(n)}_k}-e^{(n)}_{gk}u_g\right\|&
= \left\|e^{(n)}_{gk}\left(\sum_{h\in G}e^{(n)}_{gk}\alpha_h(e^{(n)}_k)u_h\right)-e^{(n)}_{gk}u_g\right\|\\
&\leq \left\|e^{(n)}_{gk}\alpha_g(e^{(n)}_k)u_h - e^{(n)}_{gk}u_h\right\|+
\sum_{h\in G, h\neq g}\left\|e^{(n)}_{gk}\alpha_h(e^{(n)}_k)u_h\right\|\\
&< \left\|\alpha_g(e^{(n)}_k)-e^{(n)}_{gk} \right\|
+ \sum_{h\in G, h\neq g}\left\|\alpha_h(e^{(n)}_k)-e^{(n)}_{hk} \right\|\\
&< \frac{1}{n}+(|G|-1)\frac{1}{n}=\frac{|G|}{n}.
\end{align*}
It follows from condition (2) above that
\[\limsup_{n\to\I}\left\|\sum_{k\in G} f_{e^{(n)}_{gk},e^{(n)}_k}- u_g\right\| \leq \limsup_{n\to\I} \frac{|G|^2}{n}=0.\]
Hence $u_g$ belongs to the closed linear span of the $f_{a,b}$, and $\alpha$ is saturated.

For $G$ compact and second countable, we are not able to describe the approximating functions $f_{a,b}$ explicitly.
(In fact, their existence is a consequence of the Stone-Weierstrass theorem.) Our proof consists in showing that
one can build approximating functions in $A\rtimes_\alpha G$ using approximating functions in
$C(G)\rtimes_{\texttt{Lt}}G$.

Suppose that $G$ is compact and $A$ is $\sigma$-unital. For an $\alpha$-invariant subalgebra $D\subseteq A$, denote by
$\gamma_D\colon G\to \Aut(C(G,D))$
the diagonal action $\gamma=\texttt{Lt}\otimes\alpha|_D$. Then $\gamma$ is conjugate to
$\texttt{Lt}\otimes\id_D$ by \autoref{prop:CGAequivisom}. Since
$\texttt{Lt}$ is saturated (see the comments after \autoref{df:saturation}), the action
$\texttt{Lt}\otimes\id_D$ is saturated by \autoref{lma:C(G,D)Sat}. We deduce that $\gamma_D$ is also
saturated.

Since $\|\cdot\|_{L^1(G,A,\alpha)}$ dominates $\|\cdot\|_{A\rtimes_\alpha G}$,
it is enough to show that the span of the functions $f_{a,b}$, with $a,b\in A$, is dense in $L^1(G,A,\alpha)$.
Denote by $\chi_E$ the characteristic function of a Borel set
$E\subseteq G$. It is a standard fact that the linear span of
\[\{x\chi_E\colon x\in A, E\subseteq G \mbox{ Borel}\}\]
is dense in $L^1(G,A,\alpha)$.
So fix $x\in A$ and a Borel subset $E\subseteq G$.
Fix $\ep>0$. Using that $\gamma_A$ is saturated, choose
$m\in\N$ and $a_1,\ldots,a_m,b_1,\ldots,b_m\in C(G,A)$ such that
\[\left\|\sum_{j=1}^m f_{a_j,b_j}-x\chi_E\right\|<\ep,\]
where the norm is taken in $C(G,A)\rtimes_\gamma G$.

Denote by $D$ the separable, $\alpha$-invariant subalgebra of $A$ generated by the set $\{a_j,b_j\colon j=1,\ldots,m\}$.
Let $\varphi\colon C(G)\to F_{\alpha}(D,A)$ be a unital equivariant homomorphism
as in the definition of the \Rp\ for $\alpha$.
Let $\psi\colon C(G,D)\to A_{\I,\alpha}$ be the equivariant homomorphism given by
\autoref{lma:LiftFromF}. Write
\[\widehat{\psi}\colon C(G,D)\rtimes_{\gamma_D}G\to A_{\I,\alpha}\rtimes_{\alpha_\I} G,\]
for the induced map at the level of the crossed products.
Under the canonical embedding
\[A_{\I,\alpha}\rtimes_{\alpha_\I} G\hookrightarrow (A\rtimes_\alpha G)_\I\]
provided by Proposition~2.1 in~\cite{Gar_RegPropCPRdim}, we will regard $\widehat{\psi}$ as a homomorphism
\[\widehat{\psi}\colon C(G,D)\rtimes_{\gamma}G\to (A\rtimes_\alpha G)_\I.\]
It is clear that
$\widehat{\psi}(f_{a_j,b_j})=f_{\psi(a_j),\psi(b_j)}$ for all $j=1,\ldots,m$, and that $\widehat{\psi}(x\chi_E)=x\chi_E$.
Hence
\begin{align*}
\left\|\sum_{j=1}^m f_{\psi(a_j),\psi(b_j)}-x\chi_E\right\|_{(A\rtimes_\alpha G)_\I}
&= \left\|\widehat{\psi}\left(\sum_{j=1}^m f_{a_j,b_j}-x\chi_E\right)\right\|_{(A\rtimes_\alpha G)_\I} \\
&\leq \left\|\sum_{j=1}^m f_{a_j,b_j}-x\chi_E\right\|<\ep.\end{align*}

To finish the proof, for $j=1,\ldots,m$, choose bounded sequences $(\psi(a_j)_n)_{n\in\N}$ and
$(\psi(b_j)_n)_{n\in\N}$ in $C(G,D)$ which represent $\psi(a_j)$ and $\psi(b_j)$,
respectively. Then
\[\eta_{A\rtimes_\alpha G}\left(\left(f_{\psi(a_j)_n,\psi(b_j)_n}\right)_{n\in\N}\right)
=f_{\psi(a_j),\psi(b_j)}.\]

It follows that for $n$ large enough, we have
\[\left\|\sum_{j=1}^m f_{\psi(a_j)_n,\psi(b_j)_n}-x\chi_E\right\|_{A\rtimes_\alpha G}<\ep,\]
showing that $\alpha$ is saturated.

The last part of the statement follows from Rieffel's original definition of saturation (Definition~7.1.4
in~\cite{Phi_Book}; see also Proposition~7.1.3 in~\cite{Phi_Book}).
\end{proof}

\begin{rem}\label{rem:CPisFixPtAlg}
In this paper, we will show that a number of properties pass from $A$ to $A^\alpha$ and $A\rtimes_\alpha G$.
These properties are all preserved by Morita equivalence. Our strategy will be to show first that the property
in question passes to the fixed point algebra. Once this is accomplished, \autoref{prop:saturated} will imply
the result for $A\rtimes_\alpha G$. An alternative to this approach is as follows: with $\lambda\colon G\to\U(L^2(G))$
denoting the left regular representation, the crossed product
$A\rtimes_\alpha G$ is isomorphic to the fixed point algebra
\[(A\otimes\K(L^2(G)))^{\alpha\otimes\Ad(\lambda)}.\]
Now, if $\alpha$ has the Rokhlin property, it is immediate to check that so does $\alpha\otimes\Ad(\lambda)$.
If the property in question has been shown to pass to fixed point algebras by Rokhlin actions and is invariant under Morita equivalence,
then it follows that it also passes to their crossed products.
\end{rem}

We recall here that if $\alpha\colon G\to\Aut(A)$ is an action of a compact group on a $\sigma$-unital \ca\ $A$, then we have
the following estimates of the nuclear dimension (Definition~2.1 in~\cite{WinZac}) and decomposition rank (Definition~3.1
in~\cite{KirWin_covdim}) of the crossed product in terms of those of $A$ and the Rokhlin dimension of $\alpha$ (Definition~3.2
in~\cite{Gar_Rdim}):
\[\dimnuc(A\rtimes_\alpha G)\leq (\dimnuc(A)+1)(\dimRok(\alpha_+1)-1,\]
and
\[\dr(A\rtimes_\alpha G)\leq (\dr(A)+1)(\dimRok(\alpha_+1)-1.\]
(For the proofs, see Theorems~3.3 and 3.4 in~\cite{Gar_RegPropCPRdim} for the case when $A$ is unital, and see
\cite{GHS_preparation} for the case of arbitrary $\sigma$-unital $A$.)

Since unital completely positive maps of order zero are necessarily homomorphisms, it is easy to see
that the Rokhlin property for a compact group action agrees with Rokhlin dimension zero in the sense of
Definition~3.2 in \cite{Gar_Rdim}. In particular, we deduce the following.

\begin{thm}\label{thm:DimNucDR}
Let $A$ be a $\sigma$-unital \ca, let $G$ be a second-countable compact group, and let
$\alpha\colon G\to\Aut(A)$ be an action with the \Rp. Then
\begin{align*} \dimnuc(A^\alpha)&=\dimnuc(A\rtimes_\alpha G)\leq \dimnuc(A), \ \mbox{ and } \\
\dr(A^\alpha)&=\dr(A\rtimes_\alpha G)\leq \dr(A).\end{align*}\end{thm}
\begin{proof}
The equalities $\dimnuc(A^\alpha)=\dimnuc(A\rtimes_\alpha G)$ and $\dr(A^\alpha)=\dr(A\rtimes_\alpha G)$
follow from \autoref{prop:saturated}, Morita equivalent \ca s have the same nuclear dimension and decomposition rank.
The two inequalities follow from the comments before this theorem, since $\dimRok(\alpha)=0$. \end{proof}

\begin{cor}\label{cor:AFpreserved}
Let $A$ be an AF-algebra, let $G$ be a second-countable compact group, and let
$\alpha\colon G\to\Aut(A)$ be an action with the \Rp. Then $A^\alpha$ and $A\rtimes_\alpha G$
are AF-algebras.\end{cor}
\begin{proof} Since a separable \ca\ has decomposition rank zero if and only if it is an
AF-algebra (Example~4.1 in \cite{KirWin_covdim}), the result follows from \autoref{thm:DimNucDR}.\end{proof}

The following result will be crucial in obtaining further structural properties for crossed products
by actions with the Rokhlin property. The proof that we present below was suggested to us by the
referee, to whom we are indebted. Our original argument was more technical and involved using certain
partitions of unity in $C(G)$ with small enough supports as in Lemma~4.2 in~\cite{Gar_RegPropCPRdim}.

\begin{thm}\label{thm:ApproxHom}
Let $A$ be a $\sigma$-unital \ca, let $G$ be a second-countable compact group, and let
$\alpha\colon G\to\Aut(A)$ be an action with the \Rp.
Given a compact subset $F_1\subseteq A$,
a compact subset $F_2\subseteq A^\alpha$ and
$\varepsilon>0$, there exists a completely positive contractive map $\psi\colon A\to A^\alpha$
such that
\be
\item For all $a,b\in F_1$, we have
\[\|\psi(ab)-\psi(a)\psi(b)\|<\ep;\]
\item For all $a\in F_2$, we have $\|\psi(a)-a\|<\ep$.\ee
Moreover, if $A$ is unital, then we can choose $\psi$ so that $\psi(1)=1$.

In particular, when $A$ is separable, there exists an approximate homomorphism
$(\psi_n)_{n\in\N}$ consisting of completely positive contractive linear maps
$\psi_n\colon A\to A^\alpha$ for $n\in\N$, which can be arranged to be unital if $A$ is,
such that $\lim\limits_{n\to\I}\|\psi_n(a)-a\|=0$ for all $a\in A^\alpha$.
\end{thm}


\begin{proof} Denote by $D$ the separable, $\alpha$-invariant subalgebra generated by $F_1\cup F_2$.
Use the Rokhlin property for $\alpha$ to choose a unital equivariant homomorphism $\varphi\colon C(G)\to F_\alpha(D,A)$.
Using Choi-Effros lifting theorem, find a lift $(\varphi_n)_{n\in\N}$ of $\varphi$ consisting of completely
positive, contractive maps $\varphi_n\colon C(G)\to A$, which must then satisfy
\be
\item[(a)] $\lim\limits_{n\to\I}\|\varphi_n(ab)d-\varphi_n(a)\varphi_n(b)d\|=0$ for all $a,b\in A$ and for all $d\in D$;
\item[(b)] $\lim\limits_{n\to\I} \|\varphi_n(1)d-d\|=0$ for all $d\in D$ (one can arrange that $\varphi_n(1)=1$ if $A$ is unital);
\item[(c)] $\lim\limits_{n\to\I} \|\varphi_n(f)d-d\varphi_n(f)\|=0$ for all $d\in D$ and for all $f\in C(G)$;
\item[(d)] $\lim\limits_{n\to\I} \sup\limits_{g\in G}\|\varphi_n(\texttt{Lt}_g(f))d-\alpha_g(\varphi_n(f))d\|=0$ for all $f\in C(G)$ and for all $d\in D$.
\ee
(In the last condition, the fact that $\|\varphi_n(\texttt{Lt}_g(f))d-\alpha_g(\varphi_n(f))d\|$ goes to
zero \emph{uniformly} on $g\in G$, and not just pointwise, follows from Dini's theorem using that the image of $\varphi$ lands in the part
of $F(D,A)$ where $G$ acts continuously; see \autoref{df:CtralSeq}.)

Denote by $\mu$ the normalized Haar measure on $G$. For $n\in\N$, define $\theta_n\colon C(G)\to A$ by
\[\theta_n(f)=\int_G \alpha_g(\varphi_n(\texttt{Lt}_{g^{-1}}(f))) \ d\mu(g)\]
for $f\in C(G)$. It is clear that $\theta_n$ is completely positive and contractive, and it is easy to check that
it is equivariant using translation invariance of $\mu$. Fix $f\in C(G)$ and $d\in D$. We use condition (d) at the last step to obtain
\begin{align*}
\limsup_{n\to\I} \|\theta_n(f)d-\varphi_n(f)d\|
&=  \limsup_{n\to\I} \left\|\int_G \alpha_g(\varphi_n(\texttt{Lt}_{g^{-1}}(f)))d -\varphi_n(f)d \ d\mu(g)\right\|\\
&\leq  \limsup_{n\to\I}\int_G \|\alpha_g(\varphi_n(\texttt{Lt}_{g^{-1}}(f)))d -\varphi_n(f)d\| \ d\mu(g)\\
&\leq \limsup_{n\to\I} \sup\limits_{g\in G}\|\alpha_g(\varphi_n(\texttt{Lt}_{g^{-1}}(f)))d -\varphi_n(f)d\|=0.
\end{align*}
We deduce that $\lim\limits_{n\to\I} \|\theta_n(f)d-\varphi_n(f)d\|$ exists and equals zero. It follows that the map
$\theta\colon C(G)\to F_\alpha(D,A)$ that
$(\theta_n)_{n\in\N}$ determines is also a lift for $\varphi$. In particular, the maps $\theta_n$ satisfy conditions (a), (b) and (c) above,
while condition (d) is satisfied exactly for each $\theta_n$.

Now, for $n\in\N$, define $\rho_n\colon C(G)\otimes A\to A$ by
\[\rho_n(f\otimes a)=\theta_n(f^{\frac{1}{2}}) a\theta_n(f^{\frac{1}{2}})\]
for $f\in C(G)$ with $f\geq 0$ (and extended linearly), and for all $a\in A$. Then $\rho_n$ is completely positive and contractive. It is also equivariant,
since for $f\in C(G)_+$ and $a\in A$, we have
\begin{align*}
\alpha_g(\rho_n(f\otimes a))&=\alpha_g\left(\theta_n(f^{\frac{1}{2}}) a\theta_n(f^{\frac{1}{2}})\right)\\
&= \alpha_g\left(\theta_n(f^{\frac{1}{2}})\right) \alpha_g(a) \alpha_g\left(\theta_n(f^{\frac{1}{2}})\right)\\
&=\theta_n(\texttt{Lt}_g(f^{\frac{1}{2}})) \alpha_g(a)\theta_n(\texttt{Lt}_g(f^{\frac{1}{2}}))\\
&= \rho_n\left(\texttt{Lt}_g(f)\otimes \alpha_g(a)\right)
\end{align*}
for all $g\in G$. Observe also that
\[\lim_{n\to\I}\|\rho_n(f\otimes d)-\theta_n(f)d\|=0\]
for all $f\in C(G)$ (not just for $f\geq 0$) and for all $d\in D$, by condition (c) above. In particular, for $f_1,f_2\in C(G)$ and
$d_1,d_2\in D$, we use condition (a) above applied to $\theta_n$ to deduce that
\begin{align*}
&\limsup_{n\to\I}\|\rho_n(f_1f_2\otimes d_1d_2)-\rho_n(f_1\otimes d_1)\rho_n(f_2\otimes d_2)\|\\
& \ \ \ \ \ \ \ = \limsup_{n\to\I}\|\theta_n(f_1f_2) d_1d_2-\theta_n(f_1) d_1\theta_n(f_2) d_2\|\\
&  \ \ \ \ \ \ \ =  \limsup_{n\to\I}\|\left[\theta_n(f_1f_2)-\theta_n(f_1)\theta_n(f_2)\right] d_1d_2\|=0.
\end{align*}
It follows that the restrictions of the maps $\rho_n$ to $C(G)\otimes D$ determine an asymptotically
multiplicative map $C(G)\otimes D\to A$.

By taking fixed point algebras in the conclusion of \autoref{prop:CGAequivisom}, we deduce that
there is an isomorphism $\sigma\colon A\to C(G,A)^{\gamma_A}$
given by $\sigma(a)(g)=\alpha_g(a)$ for $a\in A$ and $g\in G$. In particular, and under the identification of $C(G,A)$ with
$C(G)\otimes A$, the isomorphism $\sigma$ satisfies
$\sigma(a)=1\otimes a$ for all $a\in A^{\alpha}$. For $n\in\N$, let $\psi_n\colon A\to A^\alpha$
be given by $\psi_n=\rho_n\circ\sigma$.

Given $a,b\in F_1\subseteq D$, we have
\begin{align*} \limsup_{n\to\I}\|\psi_n(ab)-\psi_n(a)\psi_n(b)\|
&=\limsup_{n\to\I}\|\rho_n(\sigma(ab))-\rho_n(\sigma(a))\rho_n(\sigma(b))\|\\
&=\limsup_{n\to\I}\|\rho_n(\sigma(a)\sigma(b))-\rho_n(\sigma(a))\rho_n(\sigma(b))\|=0,\end{align*}
since $\sigma(a),\sigma(b)$, and $\sigma(ab)$ belong to $C(G)\otimes D$ and the maps $\rho_n$ are asymptotically multiplicative
on $C(G)\otimes D$.

Finally, given $a\in F_2\subseteq D$, we use condition (c) above for $\theta_n$ at the third step, and condition (b) at the fourth step to get
\begin{align*} \limsup_{n\to\I}\|\psi_n(a)-a\|&=\limsup_{n\to\I}\|\rho_n(\sigma(a))-a\|\\
&=\limsup_{n\to\I}\|\rho_n(1\otimes a)-a\|\\
&=\limsup_{n\to\I}\|\theta_n(1)a\theta_n(1)-a\|\\
&=\limsup_{n\to\I}\|\theta_n(1)a-a\|=0.\end{align*}

The conclusion then follows by setting $\psi=\psi_n$ for $n$ large enough. It is clear that the $\psi_n$ are unital if the $\theta_n$
are, which can always be arranged if $A$ is unital.

\end{proof}

\begin{rem}\label{rem:CommDiagrApprHom}
Adopt the notation from the theorem above. Then there is a commutative diagram
\beqa\xymatrix{A^\alpha\ar@{^{(}->}[dr]\ar[rr]^-{\id_{A^\alpha}} & & (A^\alpha)_\I. \\
& A\ar[ur]_-{\psi}}\eeqa
When $A$ is nuclear, Choi-Effros lifting theorem shows that the existence of a commutative
diagram as above is in fact equivalent to the conclusion in \autoref{thm:ApproxHom}. In the
general case, however, the existence of such a diagram is a weaker assumption. Barlak and
Szabo have independently identified this notion (see, for example, \cite{Sza_cRp}),
and have begun a systematic study of
this concept in its own right; see \cite{BarSza_preparation}.
\end{rem}

This work consists in showing
that a number of properties for $A$ pass to $A^\alpha$ (and $A\rtimes_\alpha G$). We state
our results assuming the Rokhlin property, but we really only use the existence of a commutative
diagram as in \autoref{rem:CommDiagrApprHom}. As such, our results are valid in a more general
context, and the extra flexibility will be needed in \cite{GHS_preparation}, where we study
crossed products by more general actions.

\vspace{0.3cm}

Our first application of \autoref{thm:ApproxHom} is to the ideal structure of crossed products and fixed
point algebras. In the presence of the Rokhlin property, we can describe all ideals: they are naturally
induced by invariant ideals in the original algebra.

\begin{thm}\label{thm:IdealStruct}
Let $A$ be a $\sigma$-unital \ca, let $G$ be a second-countable compact group, and let
$\alpha\colon G\to\Aut(A)$ be an action with the \Rp.
\be\item If $I$ is an ideal in $A^\alpha$, then there exists an $\alpha$-invariant ideal $J$
in $A$ such that $I=J\cap A^\alpha$.
\item If $I$ is an ideal in $A\rtimes_\alpha G$, then there exists an $\alpha$-invariant ideal $J$
in $A$ such that $I=J\rtimes_\alpha G$.\ee
\end{thm}
\begin{proof}
(1). Let $I$ be an ideal in $A^\alpha$. Then $J=\overline{AIA}$ is an $\alpha$-invariant ideal in $A$.
We claim that $J\cap A^\alpha=I$. It is clear that $I\subseteq J\cap A^\alpha$. For the reverse
inclusion, let $x\in J\cap A^\alpha$, that is, an $\alpha$-invariant element in $\overline{AIA}$. For
$n\in\N$, choose $m_n\in\N$, elements
$a^{(n)}_1,\ldots,a^{(n)}_{m_n}, b^{(n)}_1,\ldots,b^{(n)}_{m_n}$ in $A$,
and elements $x^{(n)}_1,\ldots,x^{(n)}_{m_n}$ in $I$, such that
\[\left\|x-\sum_{j=1}^{m_n}a^{(n)}_jx^{(n)}_jb^{(n)}_j\right\|<\frac{1}{n}.\]

Set $M_n=\max\limits_{j=1,\ldots,m_n} \{\|a^{(n)}_j\|, \|b^{(n)}_j\|,1\}$.
Let $(\psi_n)_{n\in\N}$ be a sequence of completely positive contractive maps $\psi_n\colon A\to A^\alpha$
as in the conclusion of \autoref{thm:ApproxHom} for the choices $\ep_n=\frac{1}{nm_nM^2_n}$ and
\[F_1^{(n)}=\{a^{(n)}_j,x^{(n)}_j,b^{(n)}_j\colon j=1,\ldots,m_n\}\cup\{x\}\]
and $F_2^{(n)}=\{x^{(n)}_j\colon j=1,\ldots,m_n\}\cup\{x\}$. Then
\[\left\|\psi_n\left(\sum_{j=1}^{m_n}a^{(n)}_jx^{(n)}_jb^{(n)}_j\right)-x\right\|
<\frac{1}{n}+\frac{1}{nm_nM_n^2}\leq \frac{2}{n}\]
and
\begin{align*}
&\left\|\psi_n\left(\sum_{j=1}^{m_n}a^{(n)}_jx^{(n)}_jb^{(n)}_j\right)-\sum_{j=1}^{m_n}
\psi_n(a^{(n)}_j)x^{(n)}_j\psi_n(b^{(n)}_j)\right\| \\
& \ \ \ \ \leq
\frac{1}{n}+\left\|\psi_n\left(\sum_{j=1}^{m_n}a^{(n)}_jx^{(n)}_jb^{(n)}_j\right)-\sum_{j=1}^{m_n}
\psi_n(a^{(n)}_j)\psi_n(x^{(n)}_j)\psi_n(b^{(n)}_j)\right\|\\
& \ \ \ \ \leq
\frac{1}{n}+\frac{1}{nM_n}+\left\|\psi_n\left(\sum_{j=1}^{m_n}a^{(n)}_jx^{(n)}_jb^{(n)}_j\right)-\sum_{j=1}^{m_n}
\psi_n(a^{(n)}_jx^{(n)}_j)\psi_n(b^{(n)}_j)\right\|\\
& \ \ \ \ \leq
\frac{1}{n}+\frac{2}{nM_n}\leq \frac{3}{n}.
\end{align*}

We conclude that
\[\left\|x-\sum_{j=1}^{m_n}
\psi_n(a^{(n)}_j)x^{(n)}_j\psi_n(b^{(n)}_j)\right\|<\frac{5}{n}.\]

Since $\sum\limits_{j=1}^{m_n}
\psi_n(a^{(n)}_j)x^{(n)}_j\psi_n(b^{(n)}_j)$ belongs to $I$, it follows that $x$ is a limit of elements in $I$,
and hence it belongs to $I$ itself.

(2). This follows from (1) together with the fact that $\alpha$ is saturated
(see \autoref{prop:saturated}). Alternatively, use \autoref{rem:CPisFixPtAlg} together
with the fact that the ideals in $A\otimes\K(L^2(G))$ have the form $I\otimes\K(L^2(G))$ for some ideal $I$
in $A$.
\end{proof}

\begin{cor} \label{cor:Simple}
Let $A$ be a $\sigma$-unital \ca, let $G$ be a second-countable compact group, and let
$\alpha\colon G\to\Aut(A)$ be an action with the \Rp. If $A$ is simple, then so are $A^\alpha$ and
$A\rtimes_\alpha G$.\end{cor}

In the following corollary, hereditary saturation is as in Definition~7.2.2 in~\cite{Phi_Book}, while
the strong Connes spectrum for an action of a non-abelian compact group (which is a subset of the set
$\widehat{G}$ of irreducible representations of the group) is as in Definition~1.2
of~\cite{GooLazPel}. (For abelian groups, the notion of strong Connes spectrum was introduced earlier by
Kishimoto in~\cite{Kis_simpleCP}.) We reproduce both definitions below for the convenience of the reader.

\begin{df}\label{df:HerSat}
Let $G$ be a compact group, let $A$ be a \ca, and let $\alpha\colon G\to\Aut(A)$ be an action. We say that
$\alpha$ is \emph{hereditarily saturated} if for every nonzero $\alpha$-invariant hereditary subalgebra
$B\subseteq A$, the restriction $\alpha|_B$ of $\alpha$ to $B$ is saturated, in the sense of \autoref{df:saturation}.
\end{df}

We need some notation, which we borrow from \cite{GooLazPel}. For an action $\alpha\colon G\to\Aut(A)$ of a
compact group $G$ on a \ca\ $A$, and for a unitary representation $\pi\colon G\to\U(\Hi_\pi)$, we set

\[A_2(\pi)=\{x\in A\otimes \K(\Hi_\pi)\colon (\alpha_g\otimes \id)(x)=x(1_A\otimes \pi_g) \ \mbox{ for all } g\in G\}.\]

We denote by $\mathrm{Her}^\alpha(A)$ the family of all nonzero $G$-invariant hereditary subalgebras of $A$.

\begin{df}\label{df:ConnSpec}
Let $G$ be a compact group, let $A$ be a \ca, and let $\alpha\colon G\to\Aut(A)$ be an action. We define the
following spectra for $\alpha$:
\be
\item \emph{Arveson spectrum:}
\[\mathrm{Sp}(\alpha)=\left\{\pi\in\widehat{G}\colon \overline{A_2(\pi)^*A_2(\pi)} \mbox{ is an essential ideal in } (A\otimes \K(\Hi_\pi))^{\alpha\otimes\Ad(\pi)}\right\}.\]
\item \emph{Strong Arveson spectrum:}
\[\widetilde{\mathrm{Sp}}(\alpha)=\left\{\pi\in\widehat{G}\colon \overline{A_2(\pi)^*A_2(\pi)} = (A\otimes \K(\Hi_\pi))^{\alpha\otimes\Ad(\pi)}\right\}.\]
\item \emph{Connes spectrum:}
\[\Gamma(\alpha)=\bigcap_{B\in \mathrm{Her}^\alpha(A)}\mathrm{Sp}(\alpha|_B).\]
\item \emph{Strong Connes spectrum:}
\[\widetilde{\Gamma}(\alpha)=\bigcap_{B\in \mathrm{Her}^\alpha(A)}\widetilde{\mathrm{Sp}}(\alpha|_B).\]
\ee
\end{df}

\begin{cor}\label{cor:HerSatConnSpec}
Let $A$ be a $\sigma$-unital \ca, let $G$ be a second-countable compact group, and let
$\alpha\colon G\to\Aut(A)$ be an action with the \Rp. Then $\alpha$ has
full strong Connes spectrum: $\widetilde{\Gamma}(\alpha)=\widehat{G}$, and it is hereditarily saturated.
\end{cor}
\begin{proof}
That $\widetilde{\Gamma}(\alpha)=\widehat{G}$ follows from Theorem~3.3 in~\cite{GooLazPel}. Hereditary
saturation of actions with full strong Connes spectrum is established in the comments after Lemma~3.1
in~\cite{GooLazPel}.
\end{proof}

\section{Generalized local approximations}

We now turn to the study of preservation of certain structural properties that have proved to be relevant in the
context of Elliott's classification program. In order to provide a conceptual approach, it will be
necessary to introduce some convenient terminology.

\begin{df}\label{df:GeneralizedLocal}
Let $\mathcal{C}$ be a class of \ca s and let $A$ be a \ca.
\be
\item We say that $A$ is an \emph{(unital) approximate $\mathcal{C}$-algebra}, if $A$ is isomorphic
to a direct limit of \ca s in $\mathcal{C}$ (with unital connecting maps).
\item We say that $A$ is a \emph{(unital) local $\mathcal{C}$-algebra}, if for every finite subset
$F\subseteq A$ and for every $\ep>0$, there exist a \ca\ $B$ in $\mathcal{C}$ and a not necessarily injective
(unital) homomorphism $\varphi\colon B\to A$ such that $\mbox{dist}(a,\varphi(B))<\ep$ for all $a\in F$.
\item We say that $A$ is a \emph{generalized (unital) local $\mathcal{C}$-algebra}, if for every finite subset
$F\subseteq A$ and for every $\ep>0$, there exist a \ca\ $B$ in $\mathcal{C}$ and sequence
$(\varphi_n)_{n\in\N}$ of asymptotically multiplicative (unital) completely positive contractive maps $\varphi_n\colon B\to A$ that
$\mbox{dist}(a,\varphi_n(B))<\ep$ for all $a\in F$ and for all $n$ sufficiently large. \ee\end{df}

\begin{rem} The term `local $\mathcal{C}$-algebra' is sometimes used to
mean that the local approximations are realized by \emph{injective} homomorphisms. For
example, in \cite{Thi_IndLimProj} Thiel says that a \ca\ $A$ is `$\mathcal{C}$-like', if for
every finite subset $F\subseteq A$ and for every $\ep>0$, there exist a \ca\ $B$ in
$\mathcal{C}$ and an \emph{injective} homomorphism $\varphi\colon B\to A$ such that
$\mbox{dist}(a,\varphi(B))<\ep$ for all $a\in F$. Finally, we point out that what
we call here `approximate $\mathcal{C}$' is called `A$\mathcal{C}$' in \cite{Thi_IndLimProj}.\end{rem}

The Rokhlin property is related to the above definition in the following way. Note that
the approximating maps for $A^\alpha$ that we obtain in the proof are not necessarily injective,
even if we assume that the approximating maps for $A$ are.

\begin{prop} \label{prop:RpAndApproxClasses}
Let $\mathcal{C}$ be a class of \ca s, let $A$ be a \ca, let $G$ be a second-countable group, and let
$\alpha\colon G\to\Aut(A)$ be an action with the \Rp. If $A$ is a (unital) local $\mathcal{C}$-algebra, then
$A^\alpha$ is a generalized (unital) local $\mathcal{C}$-algebra.\end{prop}
\begin{proof}
Let $F\subseteq A^\alpha$ be a finite subset, and let $\ep>0$. Find a \ca\ $B$ in $\mathcal{C}$ and
a (unital) homomorphism $\varphi\colon B\to A$ such that $\mbox{dist}(a,\varphi(B))<\frac{\ep}{2}$ for all $a\in F$.
Let $(\psi_n)_{n\in\N}$ be a sequence of (unital) completely positive contractive maps $\psi_n\colon A\to A^\alpha$ as in the
conclusion of \autoref{thm:ApproxHom}. Then $(\psi_n\circ\varphi)_{n\in\N}$ is a
sequence of (unital) completely positive contractive maps $B\to A^\alpha$ as in the definition of generalized
local $\mathcal{C}$-algebra.
\end{proof}

Let $\mathcal{C}$ be a class of \ca s. It is clear that any (unital) approximate $\mathcal{C}$-algebra is a
(unital) local
$\mathcal{C}$-algebra, and that any (unital) local $\mathcal{C}$-algebra is a generalized (unital)
local $\mathcal{C}$-algebra.

While the converses to these implications are known to fail in general,
the notions in \autoref{df:GeneralizedLocal} agree under fairly mild conditions on $\mathcal{C}$;
see \autoref{prop:ClassesAgree}.

\begin{df}\label{df:ApproxQuot}
Let $\mathcal{C}$ be a class of \ca s. We say that $\mathcal{C}$
has \emph{(unital) approximate quotients} if whenever $A\in\mathcal{C}$ (is unital) and $I$ is an ideal in $A$, the
quotient $A/I$ is a (unital) approximate $\mathcal{C}$-algebra, in the sense of \autoref{df:GeneralizedLocal}.
\end{df}

The term `approximate quotients' has been used in \cite{OsaPhi_CPRP} with a considerably stronger meaning.
Our weaker assumptions still yield an analog of Proposition~1.7 in \cite{OsaPhi_CPRP}; see \autoref{prop:ClassesAgree}.

We need to recall a definition due to Loring.
The original definition appears in \cite{Lor_Book}, while in Theorem~3.1 in \cite{EilLor_StableRels} it is proved
that weak semiprojectivity is equivalent to a condition that is more resemblant of semiprojectivity.
For the purposes of this paper, the original definition is better suited.

\begin{df}\label{df:wsp}
A \ca\ $A$ is said to be \emph{weakly semiprojective (in the unital category)} if given a \ca\ $B$ and
given a (unital) homomorphism $\psi\colon A\to B_\I$, there exists
a (unital) homomorphism $\varphi\colon A\to \ell^\I(\N,B)$ such that $\eta_B\circ\varphi=\psi$. In
other words, the following lifting problem can always be solved:
\begin{align*}
\xymatrix{
& \ell^\I(\N,B)\ar[d]^-{\eta_B}\\
A\ar[r]_-\psi\ar@{-->}[ur]^-\varphi & B_\I.}
\end{align*}
\end{df}

The proof of the following observation is left to the reader.
It states explicitly the formulation of weak semiprojectivity that will be used in our work,
specifically in \autoref{prop:ClassesAgree}.

\begin{rem}\label{rem:WSPinpractice}
Using the definition of the sequence algebra $B_\I$, it is easy to show that if $A$ is a
weakly semiprojective \ca, and if $(\psi_n)_{n\in\N}$
is an asymptotically $\ast$-multiplicative sequence of linear maps $\psi_n \colon A\to B$ from
$A$ to another \ca\ $B$,
then there exists a sequence $(\varphi_n)_{n\in\N}$ of homomorphisms $\varphi_n\colon A\to B$
such that
\[\lim_{n\to\I}\|\varphi_n(a)-\psi_n(a)\|=0\]
for all $a\in A$. If each $\psi_n$ is unital and $A$ is weakly semiprojective in the unital
category, then $\varphi_n$ can also be chosen to be unital.
\end{rem}

We proceed to give some examples of classes of \ca s that will be used in
\autoref{thm:RpPreservesWSPAQ}. We need a definition first, which appears as
Definition~11.2 in \cite{EilLorPed_StabAnticomm}.

\begin{df}\label{df:NCCW}
A \ca\ $A$ is said to be a \emph{one-dimensional noncommutative cellular complex}, or
one-dimensional NCCW-complex for short, if there exist finite dimensional \ca s $E$ and $F$,
and unital homomorphisms $\varphi,\psi\colon E\to F$, such that $A$ is isomorphic to the
pull back \ca
\[\{(a,b)\in E\oplus C([0,1],F)\colon b(0)=\varphi(a) \ \mbox{ and } \ b(1)=\psi(a)\}.\]\end{df}

It was shown in Theorem~6.2.2 of \cite{EilLorPed_StabAnticomm} that one-dimensional NCCW-com\-plex\-es are
semiprojective (in the unital category).

\begin{egs} \label{egs:wspaq}
The following are examples of classes of weakly semiprojective \ca s (in the unital category)
which have approximate quotients.
\be\item The class $\mathcal{C}$ of matrix algebras. The (unital) approximate
$\mathcal{C}$-algebras are precisely the matroid algebras (UHF-algebras).
\item The class $\mathcal{C}$ of finite dimensional \ca s. The (unital) approximate
$\mathcal{C}$-algebras are precisely the (unital) AF-algebras.
\item The class $\mathcal{C}$ of interval algebras, this is, algebras of the form $C([0,1])\otimes F$,
where $F$ is a finite dimensional \ca. The (unital) approximate
$\mathcal{C}$-algebras are precisely the (unital) AI-algebras.
\item The class $\mathcal{C}$ of circle algebras, this is, algebras of the form $C(\T)\otimes F$,
where $F$ is a finite dimensional \ca. The (unital) approximate
$\mathcal{C}$-algebras are precisely the (unital) A$\T$-algebras.
\item The class $\mathcal{C}$ of one-dimensional NCCW-complexes. We point out that
certain approximate $\mathcal{C}$-algebras have been classified, in terms of a variant of
their Cuntz semigroup, by Robert in \cite{Rob_classif}.
\ee\end{egs}

The following result is well-known for several particular classes.

\begin{prop}\label{prop:ClassesAgree}
Let $\mathcal{C}$ be a class of \ca s which has (unital) approximate quotients
(see \autoref{df:ApproxQuot}).
Assume that the \ca s in $\mathcal{C}$ are weakly semiprojective (in the unital category).
For a separable (unital) \ca\ $A$, the following are equivalent:
\be\item $A$ is an (unital) approximate $\mathcal{C}$-algebra;
\item $A$ is a (unital) local $\mathcal{C}$-algebra;
\item $A$ is a generalized (unital) local $\mathcal{C}$-algebra.\ee \end{prop}
\begin{proof}
The implications (1) $\Rightarrow$ (2) $\Rightarrow$ (3) are true in full generality. Weak semiprojectivity
of the algebras in $\mathcal{C}$ implies that any generalized local approximation by \ca s in $\mathcal{C}$
can be perturbed to a genuine local approximation by \ca s in $\mathcal{C}$ (see \autoref{rem:WSPinpractice}),
showing (3) $\Rightarrow$ (2).

For the implication (2) $\Rightarrow$ (1), note that since $\mathcal{C}$ has approximate quotients,
every a local $\mathcal{C}$-algebra is A$\mathcal{C}$-like, in the sense of Definition~3.2 in \cite{Thi_IndLimProj}
(see also Paragraph~3.6 there). It then follows from Theorem~3.9 in \cite{Thi_IndLimProj} that $A$ is an approximate
$\mathcal{C}$-algebra.

For the unital case, one uses \autoref{rem:WSPinpractice} to show that (3) $\Rightarrow$ (2) when units
are considered. Moreover, for (2) $\Rightarrow$ (1),
one checks that in the proof of Theorem~3.9 in \cite{Thi_IndLimProj}, if one assumes
that the building blocks are weakly semiprojective in the unital category, then the conclusion is that
a unital A$\mathcal{C}$-like algebra is a unital A$\mathcal{C}$-algebra. With the notation and terminology
of the proof of Theorem~3.9 in~\cite{Thi_IndLimProj}, suppose that $A$ is a unital $A\mathcal{C}$-like
algebra, and suppose that $\varphi\colon C\to A$ is a unital homomorphism, with $C\in\mathcal{C}$.
Since $C$ is assumed to be weakly semiprojective in the unital
category, the morphism $\alpha\colon C\to B$ can be chosen to be unital. For the same reason, one can
arrange that the morphism $\widetilde{\alpha}\colon C\to C_{k_1}$ be unital (possible by changing the
choice of $k_1$). Now, since the connecting maps $\gamma_k$ are also assumed to be unital, it is easily
seen that the one-sided approximate intertwining constructed has unital connecting maps. Finally, when
applying Proposition~3.5 in~\cite{Thi_IndLimProj}, if the algebras $A_i$, with $i\in I$, are weakly
semiprojective in the unital category, then the morphisms $\psi_k\colon A_{i(k)}\to A_{i(k+1)}$ can be
chosen to be unital as well. We leave the details to the reader.
\end{proof}

The following is the main application of our approximations results.

\begin{thm}\label{thm:RpPreservesWSPAQ}
Let $\mathcal{C}$ be a class of separable weakly semiprojective \ca s (in the unital category), and
assume that $\mathcal{C}$
has (unital) approximate quotients. Let $A$ be a (unital) local $\mathcal{C}$-algebra,
let $G$ be a second-countable group, and let $\alpha\colon G\to\Aut(A)$ be an action with the \Rp.
Then $A^\alpha$ is a (unital) approximate $\mathcal{C}$-algebra. \end{thm}
\begin{proof} This is an immediate consequence of \autoref{prop:RpAndApproxClasses} and
\autoref{prop:ClassesAgree}.\end{proof}

An alternative proof of part (2) of the corollary below is given in \autoref{cor:AFpreserved}.

\begin{cor}\label{cor:ParticularClasses}
Let $A$ be a $\sigma$-unital \ca, let $G$ be a second-countable compact group, and let
$\alpha\colon G\to\Aut(A)$ be an action with the \Rp.
\be
\item If $A$ is a matroid algebra (UHF), then $A^\alpha$ is a matroid algebra (UHF) and $A\rtimes_\alpha G$ is
a matroid algebra. (If $G$ is finite, then $A\rtimes_\alpha G$ is also a UHF-algebra.)
\item If $A$ is an AF-algebra, then so are $A^\alpha$ and $A\rtimes_\alpha G$.
\item If $A$ is an AI-algebra, then so are $A^\alpha$ and $A\rtimes_\alpha G$.
\item If $A$ is an A$\T$-algebra, then so are $A^\alpha$ and $A\rtimes_\alpha G$.
\item If $A$ is a direct limit of one-dimensional NCCW-complexes, then so are
$A^\alpha$ and $A\rtimes_\alpha G$.
\ee
\end{cor}
\begin{proof}
Since the classes in \autoref{egs:wspaq} have approximate quotients and contain only weakly semiprojective
\ca s, the claims follow from \autoref{thm:RpPreservesWSPAQ}.
\end{proof}

\autoref{thm:RpPreservesWSPAQ} allows for far more flexibility than Theorem~3.5 in \cite{OsaPhi_CPRP}, since
we do not assume our classes of \ca s to be closed under direct sums or by taking corners, nor do we
assume that our algebras are semiprojective. In particular, the class $\mathcal{C}$ of weakly semiprojective
purely infinite, simple algebras satisfies the assumptions of \autoref{thm:RpPreservesWSPAQ}, but appears not to
fit into the framework of flexible classes discussed in \cite{OsaPhi_CPRP}.

Recall that a \ca\ is said to be a \emph{Kirchberg algebra} if it is purely infinite, simple, separable
and nuclear.

The following lemma is probably standard, but we include its proof here for the sake of completeness.

\begin{lma}\label{lma:KirLimWkSj}
Let $A$ be a Kirchberg algebra satisfying the Universal Coefficient Theorem. Then $A$ is isomorphic
to a direct limit of weakly semiprojective Kirchberg algebras satisfying the Universal Coefficient Theorem.
\end{lma}
\begin{proof}
Since every non-unital Kirchberg algebra is the stabilization of a unital Kirchberg algebra, by
Proposition~3.11 in~\cite{Dad_CtsFdsFinDim}
it is enough to prove the statement when $A$ is non-unital. For $j=0,1$, set $G_j=K_j(A)$. Write
$G_j$ as a direct limit $G_j\cong \varinjlim (G_j^{(n)},\gamma_j^{(n)})$ of finitely generated abelian
groups $G^{(n)}_j$, with connecting maps
\[\gamma_j^{(n)}\colon G_j^{(n)}\to G_j^{(n+1)}.\]
For $j=0,1$, use Theorem~4.2.5 in~\cite{Phi_Classif} to find, for $n\in\N$, Kirchberg algebras
$A_n$ satisfying the Universal Coefficient Theorem with $K_j(A_n)\cong G_j^{(n)}$, and homomorphisms
\[\varphi_n\colon A_n\to A_{n+1}\]
such that $K_j(\varphi_n)$ is identified with $\gamma_j^{(n)}$ under the isomorphism $K_j(A_n)\cong G_j^{(n)}$.

The direct limit $\varinjlim (A_n, \varphi_n)$ is isomorphic to $A$ by Theorem~4.2.4 in~\cite{Phi_Classif}.
On the other hand, each of the algebras $A_n$ is weakly semiprojective by Theorem~2.2 in~\cite{Spi_WSjPI}
(see also Corollary~7.7 in~\cite{Lin_WSjPI}), so the proof is complete.
\end{proof}

\begin{thm}\label{thm:UCT}
Let $A$ be a separable, simple, nuclear \ca, let $G$ be a second-countable compact group, and let
$\alpha\colon G\to\Aut(A)$ be an action with the \Rp. If $A$ satisfies the Universal Coefficient Theorem,
then so do $A^\alpha$ and $A\rtimes_\alpha G$.
\end{thm}
\begin{proof}
We claim that it is enough to prove the statement when $A$ is a Kirchberg algebra. Indeed, a \ca\ $B$
satisfies the Universal Coefficient Theorem if and only if $B\otimes\OI$ does, since $\OI$ is $KK$-equivalent
to $\C$. On the other hand, $\alpha\otimes\id_{\OI}$ has the \Rp, and
\[(A\otimes\OI)^{\alpha\otimes\id_{\OI}}=A^\alpha\otimes\OI.\]

Suppose then that $A$ is a Kirchberg algebra. Denote by $\mathcal{C}$ the class of all unital weakly semiprojective
Kirchberg algebras satisfying the Universal Coefficient Theorem. Note that $\mathcal{C}$ has approximate quotients.
By \autoref{lma:KirLimWkSj},
$A$ is a unital approximate $\mathcal{C}$-algebra. By \autoref{thm:RpPreservesWSPAQ}, $A^\alpha$ is
also a unital approximate $\mathcal{C}$-algebra. Since the Universal Coefficient Theorem passes to direct limits,
we conclude that $A^\alpha$ satisfies it. Since $A\rtimes_\alpha G$ is Morita equivalent to $A^\alpha$, the same
holds for the crossed product.
\end{proof}

\section{Further structure results}

We now turn to preservation of classes of \ca s that are not necessarily defined in terms of
an approximation by weakly semiprojective \ca s. The classes we study can all be dealt with
using \autoref{thm:ApproxHom}.

The following is Definition~1.3 in \cite{TomWin_SSA}.

\begin{df}
A unital, separable \ca\ $\D$ is said to be \emph{strongly self-absorbing}, if it is
infinite dimensional and the map $\D\to \D\otimes_{\mathrm{min}} \D$, given by
$d\mapsto d\otimes 1$ for $d\in\mathcal{D}$, is approximately unitarily equivalent to an isomorphism.
\end{df}

It is a consequence of a result of Effros and Rosenberg that strongly self-absorbing \ca s
are nuclear, so that the choice of the tensor product in the definition above is irrelevant.
The only known examples of strongly self-absorbing \ca s are the Jiang-Su algebra $\mathcal{Z}$,
the Cuntz algebras $\Ot$ and $\mathcal{O}_\I$, UHF-algebras of infinite type, and tensor
products of $\OI$ by such UHF-algebras. It has been conjectured that these are the only
examples of strongly self-absorbing \ca s. See \cite{TomWin_SSA} for the proofs of these and other
results concerning strongly self-absorbing \ca s.

The following is a useful criterion to determine when a separable \ca\ absorbs a
strongly self-absorbing \ca\ tensorially. The proof is a straightforward combination of
Theorem~2.2 in \cite{TomWin_SSA} and the Choi-Effros lifting theorem,
and we shall omit it. (See also Proposition~4.1 in~\cite{HRW_CXalg}.)

\begin{thm}\label{thm:CriterionSSA}
Let $A$ be a separable \ca, and let $\mathcal{D}$ be a strongly self-absorbing \ca.
Then $A$ is $\mathcal{D}$-stable if and only if for every $\ep>0$, for every finite subset
$F\subseteq A$, and every finite subset $E\subseteq \mathcal{D}$, there exists a
completely positive map $\varphi\colon \mathcal{D}\to A$ such that
\be\item $\|a\varphi(d)-\varphi(d)a\|<\ep$ for all $a\in F$ and for all $d\in E$;
\item $\|\varphi(de)a-\varphi(d)\varphi(e)a\|<\ep$ for every $d,e\in E$ and every $a\in F$;
\item $\|\varphi(1)a-a\|<\ep$ for all $a\in F$.\ee\end{thm}

The following result was obtained for unital \ca s as part (1) of Corollary~3.4 in \cite{HirWin_Rp}, using
different methods. Our proof of the general case illustrates the generality of our approach.

\begin{thm}\label{thm:SSApreserved}
Let $A$ be a separable \ca, let $G$ be a second-countable compact group, and let
$\alpha\colon G\to\Aut(A)$ be an action with the \Rp. Let $\mathcal{D}$ be a strongly
self-absorbing \ca\ and assume that $A$ is $\mathcal{D}$-stable. Then $A^\alpha$ and
$A\rtimes_\alpha G$ are $\mathcal{D}$-stable as well.\end{thm}
\begin{proof} Since $\mathcal{D}$-stability is preserved under Morita equivalence
by Corollary~3.2 in \cite{TomWin_SSA}, it is enough to prove the result for $A^\alpha$.

Let $\ep>0$, and let $F\subseteq A^\alpha$ and $E\subseteq \mathcal{D}$ be finite subsets of $A$
and $\mathcal{D}$, respectively. Use \autoref{thm:CriterionSSA} to choose a
completely positive map $\varphi\colon\mathcal{D}\to A$ satisfying
\be\item $\|a\varphi(d)-\varphi(d)a\|<\ep$ for all $a\in F$ and for all $d\in E$;
\item $\|\varphi(de)a-\varphi(d)\varphi(e)a\|<\ep$ for all $d,e\in E$ and all $a\in F$;
\item $\|\varphi(1)a-a\|<\ep$ and $\|a\varphi(1)-a\|<\ep$ for all $a\in F$.
\ee

Let $(\psi_n)_{n\in\N}$ be a sequence of completely positive contractive maps $\psi_n\colon A\to A^\alpha$
as in the conclusion of \autoref{thm:ApproxHom} for $F_1=F\cup \{\varphi(1)\}$ and $F_2=F$.
Since $\lim\limits_{n\to\I}\psi_n(a)=a$ for all $a\in F$,
we deduce that
\[\limsup_{n\to\I} \left\|a\psi_n(\varphi(d))-\psi_n(\varphi(d))a\right\|\leq
\|a\varphi(d)-\varphi(d)a\|<\ep\]
for all $a\in F$ and all $d\in E$. Likewise,
\[\limsup_{n\to\I} \left\|\psi_n(\varphi(de))-\psi_n(\varphi(d))\psi_n(\varphi(e))\right\|\leq
\|\varphi(de)-\varphi(d)\varphi(e)\|<\ep\]
for all $d,e\in E$. Finally, for $a\in F$, we have
\[\limsup_{n\to\I}\|\psi_n(\varphi(1))a-a\|\leq \|\varphi(1)a-a\|<\ep\]
and
\[\limsup_{n\to\I}\|a\psi_n(\varphi(1))-a\|\leq \|a\varphi(1)-a\|<\ep.\]

We conclude that for $n$ large enough, the completely positive contractive map
$$\psi_n\circ\varphi\colon\mathcal{D}\to A^\alpha$$
satisfies conditions (1) through (3) of \autoref{thm:CriterionSSA}, showing that $A^\alpha$ is $\mathcal{D}$-stable.
\end{proof}

Similar methods allow one to prove that the property of being approximately divisible is inherited
by the crossed product and the fixed point algebra of a compact group action with the \Rp. (This
was first obtained, for \uca s, by Hirshberg and Winter as part (2) of Corollary~3.4 in \cite{HirWin_Rp}.) Our
proof is completely analogous to that of \autoref{thm:SSApreserved} (using a suitable version of
\autoref{thm:CriterionSSA}), so for the sake of brevity, we shall not present it here.

\vspace{0.3cm}

Our next goal is to show that Rokhlin actions preserve the property of having tracial rank at most one
in the simple, unital case.

We will need a definition of tracial rank zero and one. What we reproduce below are not Lin's original
definitions (Definition~2.1 in~\cite{Lin_TAI} and Definition~2.1 in~\cite{Lin_TAF}). Nevertheless, the
notions we define are equivalent in the simple case:
for tracial rank zero, this follows from Proposition~3.8 in \cite{Lin_TAF},
while the argument in the proof of said proposition can be adapted to show the corresponding result
for tracial rank one. Recall that an \emph{interval algebra} is a \ca\ of the form $C([0,1])\otimes E$,
where $E$ is a finite dimensional \ca. Such algebras have a finite presentation with stable relations;
see \cite{Lor_Book}.

\begin{df}\label{df:TAF}
Let $A$ be a simple, unital \ca. We say that $A$ has \emph{tracial rank at most one}, and write
$\mathrm{TR}(A)\leq 1$, if for every
finite subset $F\subseteq A$, for every $\ep>0$, and for every non-zero positive element $x\in A$,
there exist a projection $p\in A$, an interval algebra $B$, and a unital homomorphism
$\varphi\colon B\to pAp$, such that
\be\item $\|ap-pa\|<\ep$ for all $a\in F$;
\item $\dist(pap,\varphi(B))<\ep$ for all $a\in F$;
\item $1-p$ is Murray-von Neumann equivalent to a projection in $\overline{xAx}$.\ee

Additionally, we say that $A$ has \emph{tracial rank zero}, and write $\mathrm{TR}(A)=0$, if the \ca\ $B$ as
above can be chosen to be finite dimensional.
\end{df}

We will need the following notation. For $t\in \left(0,\frac{1}{2}\right)$, we denote by $f_t\colon [0,1]\to [0,1]$ the
continuous function that takes the value 0 on $[0,t]$, the value 1 on $[2t,1]$, and is linear on $[t,2t]$.

\begin{thm}\label{thm:TAF}
Let $A$ be a unital, separable, simple \ca\ with $\mathrm{TR}(A)\leq 1$,
let $G$ be a second-countable compact group, and let
$\alpha\colon G\to\Aut(A)$ be an action with the \Rp. Then
$A^\alpha$ is a unital, separable, simple \ca\ with $\mathrm{TR}(A^\alpha)\leq \mathrm{TR}(A)$.
If $G$ is finite, then the same holds for the crossed product $A\rtimes_\alpha G$.
\end{thm}
When $G$ is not finite (but compact), then $A\rtimes_\alpha G$ is never unital, and the definition of
tracial rank at most one only applies to unital \ca s.
\begin{proof}
Let $F\subseteq A^\alpha$ be a finite subset, let $\ep>0$ and let $x\in A^\alpha$ be a non-zero positive element.
Without loss of generality, we may assume that $\|a\|\leq 1$ for all $a\in F$, and that $\ep<1$.
Find $t\in \left(0,\frac{1}{2}\right)$ such that $(x-2t)_+$ is not zero. Set $y=(x-2t)_+$. Then $y$ belongs to
$A^\alpha$ and moreover $f_t(x)y=yf_t(x)=y$.

Using that $A$ has tracial rank zero, find an interval algebra $B$, a projection $p\in A$, a unital homomorphism
$\varphi\colon B\to pAp$, a projection $q\in \overline{yAy}$ and a partial isometry $s\in A$ such that
\bi\item $\|ap-pa\|<\frac{\ep}{9}$ for all $a\in F$;
\item $\dist(pap,\varphi(B))<\frac{\ep}{9}$ for all $a\in F$;
\item $1-p=s^*s$ and $q=ss^*$. \ei

Let $\widetilde{F}\subseteq B$ be a finite subset such that for all $a\in F$, there exists $b\in\widetilde{F}$
with $\|pap-\varphi(b)\|<\frac{\ep}{9}$.

Since $f_t(x)$ is a unit for $\overline{yAy}$, it follows that $q=f_t(x)qf_t(x)$.
As $A$ is unital and separable, we can use \autoref{thm:ApproxHom} to find an approximate homomorphism
$(\psi_n)_{n\in\N}$ from $A$ to $A^\alpha$, consisting of unital completely positive maps $\psi_n\colon A\to A^\alpha$
satisfying $\lim\limits_{n\to\I}\|\psi_n(a)-a\|=0$ for all $a\in A^\alpha$. (For example, one chooses increasing
families $\left(F_1^{(n)}\right)_{n\in\N}$ and $\left(F_2^{(n)}\right)_{n\in\N}$ of finite subsets of $A$ and $A^\alpha$,
respectively, whose union is dense in $A$ and $A^\alpha$, and obtains $\psi_n$ by applying the main part of \autoref{thm:ApproxHom}
with tolerance $\ep_n=1/n$ and sets $F_1^{(n)}\subseteq A$ and $F_2^{(n)}\subseteq A^\alpha$.) We then have
\be
\item[(a)]$\limsup\limits_{n\to\I}\left\|\psi_n(p)a-a\psi_n(p)\right\|<\frac{\ep}{9}$ for all $a\in F$;
\item[(b)]$\limsup\limits_{n\to\I} \dist\left(\psi_n(p)a\psi_n(a),(\psi_n\circ\varphi)(B)\right)<\frac{\ep}{9}$
for all $a\in F$;
\item[(c)]$\lim\limits_{n\to\I}\left\|\psi_n(p)a\psi_n(p)-\psi_n(pap)\right\|=0$;
\item[(d)]$\lim\limits_{n\to\I}\left\|\psi_n(p)^*\psi_n(p)-\psi_n(p)\right\|=0$;
\item[(e)]$\lim\limits_{n\to\I}\left\|1-\psi_n(p)-\psi_n(s)^*\psi_n(s)\right\|=0$;
\item[(f)]$\lim\limits_{n\to\I}\left\|\psi_n(q)\psi_n(s)\psi_n(1-p)-\psi_n(s)\right\|=0$;
\item[(g)]$\lim\limits_{n\to\I}\left\|\psi_n(q)^*\psi_n(q)-\psi_n(q)\right\|=0$;
\item[(h)]$\lim\limits_{n\to\I}\left\|\psi_n(q)-\psi_n(s)\psi_n(s)^*\right\|=0$;
\item[(i)]$\lim\limits_{n\to\I} \left\|\psi_n(q)-f_t(x)\psi_n(q)f_t(x)\right\|=0$.
\ee
With $r_n=f_t(x)\psi_n(q)f_t(x)$ for $n\in\N$, it follows from conditions (g) and (i) that
\be
\item[(j)] $\lim\limits_{n\to\I} \left\|r_n^*r_n-r_n\right\|=0$.
\ee
\ \\
\indent Find $\delta_1>0$ such that whenever $e$ is an element in a \ca\ $C$ such that $\|e^*e-e\|<\delta_1$, then
there exists a projection $f$ in $C$ such that $\|e-f\|<\frac{\ep}{9}$.
Fix a finite set $\mathcal{G}\subseteq B$ of generators for $B$. Using semiprojectivity of $B$ in the unital category
(specifically, the fact that the relations defining it are stable),
find $\delta_2>0$ such that whenever $D$ is a unital \ca\ and $\rho\colon B\to D$ is a unital positive linear map which
is $\delta_2$-multiplicative on $\mathcal{G}$, there exists a unital homomorphism $\pi\colon B\to D$
such that $\|\rho(b)-\pi(b)\|<\frac{\ep}{9}$ for all $b\in \widetilde{F}$. (Observe that we are not fixing the
target algebra $D$, which will later be taken to be of the form $fA^\alpha f$ for some projection $f\in A^\alpha$.)
Set $\delta=\min\{\delta_1,\delta_2\}$.

Choose $n\in\N$ large enough so that the quantities in conditions (a), (b), (c), (e) and (i) are less than
$\frac{\ep}{9}$, the quantities in (d) and (j) are less
than $\delta$, the quantities in (e) and (g) are less than $1-\ep$, and so that $\psi_n\circ\varphi$ is
$\delta$-multiplicative on $\mathcal{G}$.
Since $r_n$ belongs to $\overline{xA^\alpha x}$ for all $n\in\N$, by the choice of $\delta$ there
exist a projection $e$ in $\overline{xA^\alpha x}$ such that $\|e-r_n\|<\frac{\ep}{9}$, and a projection
$f\in A^\alpha$ such that $\|f-\psi_n(p)\|<\frac{\ep}{9}$.
Let $\pi\colon B\to fA^\alpha f$ be a unital homomorphism satisfying
\[\|\pi(b)-(\psi_n\circ\varphi)(b)\|<\frac{\ep}{9}\]
for all $b\in\mathcal{G}\cup\widetilde{F}$.

We claim that the projection $f$ and the homomorphism $\pi\colon B\to fA^\alpha f$ satisfy the conditions
in \autoref{df:TAF}. Since $\pi$ is unital, we must have $\pi(1)=f$.

Given $a\in F$, the estimate
\[\|af-fa\|\leq \|a\psi_n(p)-\psi_n(p)a\|+2\|\psi_n(p)-f\|<\frac{3\ep}{9}<\ep\]
shows that condition (1) is satisfied. In order to check condition (2), given $a\in F$, choose $b\in \widetilde{F}$ such that
\[\|pap-\varphi(b)\|<\frac{\ep}{9}.\]
Then
\begin{align*} \|faf-\pi(b)\|&\leq \|faf-\psi_n(p)a\psi_n(p)\|+\|\psi_n(p)a\psi_n(p)-\psi_n(\varphi(b))\| \\
& \ \ \ \  + \|\psi_n(\varphi(b))-\pi(b)\|\\
&< 2\|f-\psi_n(p)\|+\frac{\ep}{9}+\frac{\ep}{9}<\ep,\end{align*}
so condition (2) is also satisfied. To check condition (3), it is enough to show that $1-f$ is Murray-von Neumann
equivalent (in $A^\alpha$) to $e$. We have
\begin{align*}
\|(1-f)-\psi_n(s)^*\psi_n(s)\|&\leq \|f-\psi_n(p)\|+\|1-\psi_n(p)-\psi_n(s)^*\psi_n(s)\|\\
&<\frac{\ep}{9}+1-\ep=1-\frac{8\ep}{9},\end{align*}
and likewise, $\|e-\psi_n(s)\psi_n(s)^*\|<\frac{\ep}{9}+1-\ep$. On the other hand, we use the approximate versions
of equation (i) at the second step, and that of equation (f) at the third step, to get
\begin{align*}
\|\psi_n(s)-e\psi_n(s)(1-f)\|&< \frac{2\ep}{9}+ \|\psi_n(s)-f_t(x)\psi_n(q)f_t(x)\psi_n(s)\psi_n(1-p)\|\\
&<\frac{3\ep}{9}+\|\psi_n(s)-\psi_n(q)\psi_n(s)\psi_n(1-p)\|\\
&< \frac{4\ep}{9}.\end{align*}
Now, it is immediate that
\begin{align*}
\|(1-f)-(e\psi_n(s)(1-f))^*(e\psi_n(s)(1-f))\| &< 2\|\psi_n(s)-e\psi_n(s)(1-f)\| \\
& \ \ \ \ + \|(1-f)-\psi_n(s)^*\psi_n(s)\|\\
&< \frac{8\ep}{9}+1-\frac{8\ep}{9}=1.
\end{align*}
Likewise,
\[\|e-(e\psi_n(s)(1-f))(e\psi_n(s)(1-f))^*\|<1.\]

By Lemma~2.5.3 in \cite{Lin_Book} applied to $e\psi_n(s)(1-f)$, we conclude that $1-f$ and $e$ are
Murray-von Neumann equivalent in $A^\alpha$, and the proof of the first part of the statement is complete.

It is clear that if $A$ has tracial rank zero and we choose $B$ to be finite dimensional, then the above
proof shows that $A^\alpha$ has tracial rank zero as well.

Finally, if $G$ is finite, then the last claim of the statement follows from the fact that $A^\alpha$ and
$A\rtimes_\alpha G$ are Morita equivalent.
\end{proof}

We believe that a condition weaker than the Rokhlin property ought to suffice for the conclusion
of \autoref{thm:TAF} to hold. In view of Theorem~2.8 in \cite{Phi_tracialFirst}, we presume that
fixed point algebras by a suitable version of the tracial Rokhlin property for compact group
actions would preserve the class of simple \ca s with tracial rank zero.

We present two consequences of \autoref{thm:TAF}. The first one is to simple AH-algebras of slow
dimension growth and real rank zero, which do not a priori fit into the general framework of
\autoref{thm:RpPreservesWSPAQ}, since the building blocks are not necessarily weakly semiprojective.

\begin{cor}
Let $A$ be a simple, unital AH-algebra with slow dimension growth and real rank zero. Let $G$ be
a second-countable compact group, and let $\alpha\colon G\to\Aut(A)$ be an action with the Rokhlin
property. Then $A^\alpha$ is a simple, unital AH-algebra with slow dimension growth and real rank
zero.
\end{cor}
\begin{proof}
By Proposition~2.6 in~\cite{Lin_TAF}, $A$ has tracial rank zero. Thus $A^\alpha$ is a simple \ca\
with tracial rank zero by \autoref{thm:TAF}. It is clearly separable, unital, and nuclear. Moreover,
it satisfies the Universal Coefficient Theorem by \autoref{thm:UCT}.
Since AH-algebras of slow dimension growth and real rank zero exhaust
the Elliott invariant of \ca s with tracial rank zero, Theorem~5.2 in~\cite{Lin_ClassifTAF} implies
that $A^\alpha$ is an AH-algebra with slow dimension growth and real rank zero.
\end{proof}

Denote by $\mathcal{Q}$ the universal UHF-algebra.
Recall that a simple,
separable, \uca\ $A$ is said to have \emph{rational tracial rank
at most one}, if $\mathrm{TR}(A\otimes \mathcal{Q})\leq 1$
(see Definition~11.8 in~\cite{Lin_Asymptotic}, and see the comments after it for examples of algebras
of rational tracial rank at most one).

\begin{cor} \label{cor:RationallyTAI}
Let $A$ be a simple, separable, \uca, let $G$ be a second-countable compact group, and let
$\alpha\colon G\to\Aut(A)$ be an action with the \Rp. If $A$ has rational tracial rank at most one,
then so does $A^\alpha$ (and also $A\rtimes_\alpha G$ if $G$ is finite).
\end{cor}
\begin{proof}
The result is an immediate consequence of \autoref{thm:TAF} applied to the action
$\alpha\otimes\id_{\mathcal{Q}}\colon G\to \Aut(A\otimes\mathcal{Q})$.
\end{proof}

We now turn to pure infiniteness in the non-simple case. The following is Definition~4.1 in
\cite{KirRor_absOI}

\begin{df}\label{df:PI}
A \ca\ $A$ is said to be \emph{purely infinite} if the following conditions are satisfied:
\be\item There are no non-zero characters (this is, homomorphisms onto the complex numbers) on $A$, and
\item For every pair $a,b$ of positive elements in $A$, with $b$ in the ideal generated by $a$, there
exists a sequence $(x_n)_{n\in\N}$ in $A$ such that $\lim\limits_{n\to\I}\|x_n^*bx_n-a\|=0$. \ee
\end{df}

The following is Theorem~4.16 in \cite{KirRor_absOI} (see also Definition~3.2 in \cite{KirRor_absOI}).

\begin{thm} \label{thm:EquivPI}
Let $A$ be a \ca. Then $A$ is purely infinite if and only if for every nonzero positive
element $a\in A$, we have $a\oplus a\preceq a$.\end{thm}

We use the above result to show that, in the presence of the \Rp, pure infiniteness is inherited by
the fixed point algebra and the crossed product.

\begin{prop} \label{prop:PI}
Let $A$ be a $\sigma$-unital \ca, let $G$ be a second-countable compact group, and let
$\alpha\colon G\to\Aut(A)$ be an action with the \Rp. If $A$ is purely infinite, then so are
$A^\alpha$ and $A\rtimes_\alpha G$.
\end{prop}
\begin{proof}
By \autoref{prop:saturated} (see also \autoref{rem:CPisFixPtAlg}) and Theorem~4.23
in~\cite{KirRor_absOI}, it is enough to prove
the result for $A^\alpha$. Let $a$ be a nonzero positive element in $A^\alpha$. Since $A$ is
purely infinite, by Theorem~4.16 in \cite{KirRor_absOI} (here reproduced as \autoref{thm:EquivPI}),
there exist sequences $(x_n)_{n\in\N}$ and $(y_n)_{n\in\N}$ in $A$ such that
\be
\item[(a)]$\lim\limits_{n\to\I} \|x_n^*ax_n-a\|=0$;
\item[(b)]$\lim\limits_{n\to\I} \|x_n^*ay_n\|=0$;
\item[(c)]$\lim\limits_{n\to\I} \|y_n^*ax_n\|=0$;
\item[(d)]$\lim\limits_{n\to\I} \|y_n^*ay_n-a\|=0$.
\ee

Let $(\psi_n)_{n\in\N}$ be a sequence of completely positive contractive maps $\psi_n\colon A\to A^\alpha$
as in the conclusion of \autoref{thm:ApproxHom}. Easy applications of the triangle inequality yield
\be
\item[(a')]$\lim\limits_{n\to\I} \|\psi_n(x_n)^*a\psi_n(x_n)-a\|=0$;
\item[(b')]$\lim\limits_{n\to\I} \|\psi_n(x_n)^*a\psi_n(y_n)\|=0$;
\item[(c')]$\lim\limits_{n\to\I} \|\psi_n(y_n)^*a\psi_n(x_n)\|=0$;
\item[(d')]$\lim\limits_{n\to\I} \|\psi_n(y_n)^*a\psi_n(y_n)-a\|=0$.
\ee

Since $\psi_n(x_n)$ and $\psi_n(y_n)$ belong to $A^\alpha$ for all $n\in\N$, we conclude
that $a\oplus a\preceq a$ in $A^\alpha$. It now follows from
Theorem~4.16 in \cite{KirRor_absOI} (here reproduced as \autoref{thm:EquivPI}) that
$A^\alpha$ is purely infinite, as desired.
\end{proof}

\begin{cor} \label{cor:Kir}
Let $A$ be a Kirchberg algebra, let $G$ be a second-countable compact group, and let
$\alpha\colon G\to\Aut(A)$ be an action with the \Rp. Then $A^\alpha$ and $A\rtimes_\alpha G$ are Kirchberg
algebras.\end{cor}
\begin{proof}
It is well-known that $A^\alpha$ and $A\rtimes_\alpha G$ are nuclear and separable. Simplicity follows from
\autoref{cor:Simple}, and pure infiniteness follows from \autoref{prop:PI}.
\end{proof}

For the sake of comparison, we mention here that stable finiteness passes to fixed point algebras and crossed
products by arbitrary compact group actions, since we have $A^\alpha\subseteq A$ and $A\rtimes_\alpha G\subseteq A\otimes\K(L^2(G))$,
and stable finiteness passes to subalgebras.

The following definition is standard.

\begin{df}
Let $A$ be a \ca.
\be
\item If $A$ is unital, we say that it has \emph{real rank zero} if the set of invertible
self adjoint elements in $A$ is dense in the set of self adjoint elements. If $A$ is not
unital, we say that it has real rank zero if so does its unitization $\widetilde{A}$.
\item If $A$ is unital, we say that it has \emph{stable rank one} if the set of invertible
elements is dense in $A$. If $A$ is not unital, we say that it has stable rank one if so
does its unitization $\widetilde{A}$.\ee
\end{df}

In the following proposition, the Rokhlin property is surely stronger than necessary for the
conclusion to hold, although some condition on the action must be imposed. We do not know, for
instance, whether finite Rokhlin dimension with commuting towers preserves real rank zero and
stable rank one.

\begin{prop}\label{prop:rrz_sro}
Let $A$ be a $\sigma$-unital \ca, let $G$ be a second-countable compact group, and let
$\alpha\colon G\to\Aut(A)$ be an action with the \Rp.
\be\item If $A$ has real rank zero, then so do $A^\alpha$ and $A\rtimes_\alpha G$.
\item If $A$ has stable rank one, then so do $A^\alpha$ and $A\rtimes_\alpha G$.\ee\end{prop}
\begin{proof} By \autoref{prop:saturated} (see also \autoref{rem:CPisFixPtAlg}),
Theorem~3.3 in~\cite{Rie_dimensionSR}, and Theorem~2.5 in
\cite{BroPed}, it is enough to prove the proposition for $A^\alpha$.
Since the proofs of both parts are similar, we only prove the first one.

Since the commutative diagram in \autoref{rem:CommDiagrApprHom} can be unitized, it is enough
to assume that $A$ is unital. (Equivalently, extend the linear maps $\psi_n\colon A\to A^\alpha$
in the conclusion of \autoref{thm:ApproxHom} to unital maps $\widetilde{\psi_n}\colon \widetilde{A}
\to \widetilde{A}^\alpha$.)

Let $a$ be a self-adjoint element in $A^\alpha$ and let $\ep>0$. Since $A$ has real rank zero,
there exists an invertible self-adjoint element $b$ in $A$ such that $\|b-a\|<\frac{\ep}{2}$. Let
$(\psi_n)_{n\in\N}$ be a sequence of unital completely positive maps $A\to A^\alpha$ as in
the conclusion of \autoref{thm:ApproxHom}. Then $\psi_n(b)$ is self-adjoint for all $n\in\N$.
Moreover,
\begin{align*} \lim_{n\to\I}\left\|\psi_n(b)\psi_n(b^{-1})-1\right\|&=
\lim_{n\to\I}\left\|\psi_n(b^{-1})\psi_n(b)-1\right\|=0\ \mbox{ and}\\
\lim_{n\to\I}\left\|\psi_n(a)-a\right\|&=0.\end{align*}
Choose $n$ large enough so that
\[\left\|\psi_n(b)\psi_n(b^{-1})-1\right\|<1 \mbox{ and } \left\|\psi_n(b^{-1})\psi_n(b)-1\right\|<1,\]
and also so that $\left\|\psi_n(a)-a\right\|<\frac{\ep}{2}$.
Then $\psi_n(b)\psi_n(b^{-1})$ and $\psi_n(b^{-1})\psi_n(b)$
are invertible, and hence so is $\psi_n(b)$. Finally,
\[\|a-\psi_n(b)\|\leq \|a-\psi_n(a)\|+\|\psi_n(a)-\psi_n(b)\|<\frac{\ep}{2}+\frac{\ep}{2}=\ep,\]
which shows that $A^\alpha$ has real rank zero.
\end{proof}

We now turn to traces. For a trace $\tau$ on a \ca\ $A$, we also denote by $\tau$ its amplification
to any matrix algebra $M_n(A)$. We denote by $T(A)$ the set of all tracial states on $A$.

The following is one of Blackadar's fundamental comparability questions:

\begin{df}\label{df:OrdPjnTraces}
Let $A$ be a simple unital \ca. We say the the \emph{order on projections (in $A$) is determined by traces},
if whenever $p$ and $q$ are projections in $M_\I(A)$ satisfying $\tau(p)\leq \tau(q)$ for all $\tau \in T(A)$,
then $p\precsim_{M-vN} q$.
\end{df}

The following extends, with a simpler proof, Proposition~4.8 in~\cite{OsaPhi_CPRP}.

\begin{prop}\label{prop:OrdPjnTraces}
Let $A$ be a simple \uca, and suppose that the order on its projections is determined by traces.
Let $G$ be a second-countable compact group, and let $\alpha\colon G\to\Aut(A)$ be an action with the Rokhlin
property. Then the order on projections in $A^\alpha$ is determined by traces.
\end{prop}
\begin{proof}
Since $\alpha\otimes \id_{M_n}\colon G\to\Aut(A\otimes M_n)$ has the Rokhlin property and
$(A\otimes M_n)^{\alpha\otimes \id_{M_n}}=A^\alpha\otimes M_n$, in \autoref{df:OrdPjnTraces} it is
enough to consider projections in the algebra.

Let $p$ and $q$ be projections in $A^\alpha$, and suppose that it is not the case that $p\precsim_{M-vN} q$
in $A^\alpha$.
We want to show that there exists a tracial state $\tau$ on $A^\alpha$ such that $\tau(p)\geq \tau(q)$. By
part~(1) in Proposition~3.2 in~\cite{Gar_CptRok}, it is not the case that $p\precsim_{M-vN} q$
in $A$, so there exists a tracial state $\omega$ on $A$ such that $\omega(p)\geq \omega(q)$.
Now take $\tau=\omega|_{A^\alpha}$.
\end{proof}

Finally, we close this section by exploring the extent to which semiprojectivity
passes from $A$ to the fixed point algebra and the crossed product by a
compact group with the Rokhlin property. Even though we have not been able to answer this
question for semiprojectivity, we can provide a satisfying answer for \emph{weak} semiprojectivity
(see \autoref{df:wsp}).

In order to show this, we introduce the following technical definition, which is inspired
in the notion of ``corona extendibility" (Definition~1.1 in~\cite{LorPed_CoronaExt}; we are thankful
to Hannes Thiel for providing this reference).

\begin{df}\label{df:SeqAlgExt}
Let $\theta\colon A\to B$ be a homomorphism between \ca s $A$ and $B$. We say that $\theta$ is
\emph{sequence algebra extendible}, if whenever $E$ is a \ca\ and $\varphi\colon A\to E_\I$ is a
homomorphism, there exists a homomorphism $\rho\colon B\to E_\I$ such that
$\varphi=\rho\circ\theta$.
\end{df}

In analogy with Lemma~1.4 in~\cite{LorPed_CoronaExt}, we have the following:

\begin{lma}\label{lma:SeqAlgExtWkSj}
Let $\theta\colon A\to B$ be a sequence algebra extendible homomorphism between \ca s $A$ and $B$.
If $B$ is weakly semiprojective, then so is $A$.
\end{lma}
\begin{proof}
This is straightforward.
\end{proof}

The following lemma will allow us to replace maps from separable $C^*$-algebras into $(E_\I)_\I$ with maps into $E_\I$.
Its proof boils down to a more or less standard reindexation argument.

\begin{lma}\label{lma:Einftyinfty}
Let $A$ and $B$ be separable \ca s, let $E$ be a \ca. Denote by $j\colon E_\I\to (E_\I)_\I$ the canonical
embedding as constant sequences. Suppose that we are given homomorphisms $\theta\colon A\to B$, $\varphi\colon A\to E_\I$ and
$\psi\colon B\to (E_\I)_\I$ making the following diagram commute:
\beqa
\xymatrix{
A\ar[r]^{\theta}\ar[d]_{\varphi}& B\ar[d]^-{\psi}\ar@{-->}[dl]^-{\rho}\\
E_\I \ar[r]_-{j} & (E_\I)_\I
.}
\eeqa

Then there exists a homomorphism $\rho\colon B\to E_\I$ such that $\rho\circ\theta=\varphi$.
\end{lma}
\begin{proof}
Let $(\psi^{(n)})_{n\in\N}$ be a sequence of linear maps $\psi^{(n)}\colon B\to E_\I$ lifting $\psi$. For $n\in\N$,
let $(\psi^{(n)}_m)_{m\in\N}$ be a sequence of linear maps $\psi^{(n)}_m\colon B\to E$ lifting $\psi^{(n)}$.
Let also $(\varphi_k)_{k\in\N}$ be a sequence of linear maps $\varphi_k\colon A\to E$ lifting $\varphi$. With the
natural representation of elements in $(E_\I)_\I$ by doubly indexed sequences in $E$, the identity
$\psi\circ\theta=j\circ\varphi$ can be rephrased as
\begin{align*}
\lim_{n\to\I}\limsup_{m\to\I} \left\|\psi^{(n)}_m(\theta(a))-\varphi_m(a) \right\|=0
\end{align*}
for all $a\in A$.
Let $(F_r)_{r\in\N}$ and $(G_r)_{r\in\N}$ be sequences of finite subsets of $A$ and $B$, respectively, such that
$\bigcup\limits_{r\in\N} F_r$ is dense in $A$ and $\bigcup\limits_{r\in\N} G_r$ is dense in $B$. Without loss of
generality, we may assume that $F_r^*=F_r$ and $F_k^2\subseteq F_{r+1}$ for all $r\in\N$, and similarly with the
sets $G_r$ for $r\in\N$. Likewise, we may assume that $\theta(F_r)\subseteq G_r$ for all $r\in\N$.

For each $r\in\N$, find an integer $n_r$ such that
\be
\item $\limsup\limits_{m\to\I} \left\|\psi^{(n_r)}_m(\theta(a))-\varphi_m(a) \right\|<\frac{1}{r}$ for all $a\in F_r$;
\item $\limsup\limits_{m\to\I} \left\|\psi^{(n_r)}_m(b^*c)-\psi^{(n_r)}_m(b)^*\psi^{(n_r)}_m(c) \right\|<\frac{1}{r}$ for all $b,c\in G_r$;
\item $\limsup\limits_{m\to\I} \left\|\psi^{(n_r)}_m(b)\right\|<\|b\|+\frac{1}{r}$ for all $b\in G_r$.
\ee
Without loss of generality, we may assume that $n_{r+1}>n_r$ for all $r\in\N$.
Similarly, find an
increasing sequence $(m_r)_{r\in\N}$ in $\N$ satisfying
\be
\item[(1')] $\left\|\psi^{(n_r)}_{m_r}(\theta(a))-\varphi_{m_r}(a) \right\|<\frac{1}{r}$ for all $a\in F_r$;
\item[(2')] $\left\|\psi^{(n_r)}_{m_r}(b^*c)-\psi^{(n_r)}_{m_r}(b)^*\psi^{(n_r)}_{m_r}(c) \right\|<\frac{1}{r}$ for all $b,c\in G_r$;
\item[(3')] $\left\|\psi^{(n_r)}_{m_r}(b)\right\|<\|b\|+\frac{1}{r}$ for all $b\in G_r$.
  \ee

Recall that $\eta_E\colon \ell^\I(\N,E)\to E_\I$ denotes the canonical quotient map.
Define $\rho\colon B \to \ell^\I(\N,E)$ by $\rho(b)=\eta_E\left(\psi^{(n_r)}_{m_r}(b)\right)_{r\in\N}$ for $b\in \N$. (One first
defines $\rho$ on the union of the $G_r$, and since it is multiplicative and contractive by construction, it extends to a homomorphism
from all of $B$.) Since the identity $\rho\circ \theta=\varphi$ holds on a dense subspace of $A$, it holds on all of $A$. This finishes
the proof.
\end{proof}

In the next proposition, we show that weak semiprojectivity passes to fixed point algebras of actions
with the Rokhlin property (and to crossed products, whenever the group is finite). Our conclusions
seem not to be obtainable with the methods developed in \cite{OsaPhi_CPRP}, since it is not in general
true that a corner of a weakly semiprojective \ca\ is weakly semiprojective.

\begin{prop}\label{prop:WksjPreserved}
Let $G$ be a second-countable compact group, let $A$ be a separable \ca, and let $\alpha\colon G\to\Aut(A)$
be an action with the Rokhlin property. Then the canonical inclusion $\iota\colon A^\alpha\to A$
is sequence algebra extendible (\autoref{df:SeqAlgExt}).

In particular, if $A$ is weakly semiprojective,
then so is $A^\alpha$ by \autoref{lma:SeqAlgExtWkSj}. If in addition $G$ is finite, then $A\rtimes_\alpha G$ is also weakly
semiprojective.\end{prop}
\begin{proof}
Use \autoref{thm:ApproxHom} to choose a sequence $(\psi_n)_{n\in\N}$ of asymptotically $\ast$-mul\-ti\-pli\-ca\-tive
linear maps $\psi_n\colon A\to A^\alpha$ such that
$\lim\limits_{n\to\I}\|\psi_n(a)-a\|=0$ for all $a\in A^\alpha$. Regard
$(\psi_n)_{n\in\N}$ as a homomorphism $\psi\colon A\to (A^\alpha)_\I$ such that the restriction $\psi|_{A^\alpha}$
agrees with the canonical inclusion $A^\alpha\hookrightarrow (A^\alpha)_\I$.

Let $E$ be a \ca\ and let $\varphi\colon A^\alpha\to E_\I$ be a homomorphism. Denote by
\[\varphi_\I\colon (A^\alpha)_\I\to (E_\I)_\I\]
the homomorphism induced by $\varphi$. There is a commutative diagram
\beqa\xymatrix{A^\alpha \ar[r]^-{\iota} \ar[d]_-{\varphi}& A\ar@{-->}[dl]^-{\rho} \ar[r]^-{\psi} &(A^\alpha)_\I \ar[dl]^-{\varphi_\I}\\
E_\I\ar[r]_-{j} & (E_\I)_\I. &}\eeqa

By \autoref{lma:Einftyinfty}, there exists a homomorphism $\rho\colon A\to E_\I$ such that $\varphi=\rho\circ \iota$.
Thus $\iota$ is sequence algebra extendible, as desired.

If $G$ is finite, then
$A\rtimes_\alpha G$ can be canonically identified with $M_{|G|}\otimes A^\alpha$, and hence it is
also weakly semiprojective.
\end{proof}

Finally, we point out that weak semiprojectivity does not in general pass to crossed products by Rokhlin actions when
the group is compact but not finite. Indeed, $C(\T)\rtimes_{\texttt{Lt}}\T\cong \K(L^2(\T))$ is not weakly semiprojective,
while $C(\T)$ is even semiprojective.


\newcommand{\etalchar}[1]{$^{#1}$}
\providecommand{\bysame}{\leavevmode\hbox to3em{\hrulefill}\thinspace}
\providecommand{\MR}{\relax\ifhmode\unskip\space\fi MR }
\providecommand{\MRhref}[2]{%
  \href{http://www.ams.org/mathscinet-getitem?mr=#1}{#2}
}
\providecommand{\href}[2]{#2}

\end{document}